\documentclass{amsart}
\usepackage{amsmath,mathrsfs,amssymb,amsthm,amscd}
\usepackage[]{fontenc}
\usepackage[all]{xy}
\usepackage{hyperref}
\usepackage{graphics}
\usepackage{color}
\usepackage{palatino}
\newcounter{EQNR}
\renewcommand{\contentsname}{Contents\\
{\footnotesize\normalfont(A table of contents should normally not
be included)}}

 \theoremstyle{plain}
 \newtheorem{thm}{Theorem}[section]
 
 \numberwithin{equation}{section} %% Comment out for sequentially-numbered
 \theoremstyle{plain}
 \theoremstyle{plain}
 \theoremstyle{definition}
 \newtheorem{defn}[thm]{Definition}
 
 \theoremstyle{plain}
 \newtheorem{prop}[thm]{Proposition}
 
 \newtheorem{lem}[thm]{Lemma}
 
 \newtheorem*{cor*}{Corollary}
 \newtheorem*{conj*}{Conjecture}
 \newtheorem*{thm*}{Theorem}

\newcommand{\bl}{\begin{lem}}
\newcommand{\el}{\end{lem}}

\newcommand{\bml}{\begin{multline}}
\newcommand{\eml}{\end{multline}}

\newcommand{\beq}{\begin{equation}}
\newcommand{\eeq}{\end{equation}}
\newcommand{\bp}{\begin{prop}}
\newcommand{\ep}{\end{prop}}
\newcommand{\bd}{\begin{defn}}
\newcommand{\ed}{\end{defn}}
\newcommand{\pf}{\begin{proof}}
\newcommand{\epf}{\end{proof}}

\newcommand{\field}[1]{\ensuremath{\mathbb{#1}}}
\newcommand{\F}{\mathcal{F}}
\newcommand{\CC}{\field{C}}
\newcommand{\df}{\equiv}
\newcommand{\NN}{\field{N}}

\newcommand{\HH}{\field{H}}

\newcommand{\RR}{\field{R}}

\newcommand{\ZZ}{\field{Z}}

\newcommand{\pol}{\mathcal{N}}
\newcommand{\sml}{\mathcal{A}}

\DeclareMathOperator{\li}{li}
\DeclareMathOperator{\supp}{supp}

\DeclareMathOperator{\PSL}{PSL}

\DeclareMathOperator{\R}{Re}
\DeclareMathOperator{\I}{Im}

\DeclareMathOperator{\tr}{tr}
\DeclareMathOperator{\bo}{\mathcal{O}}

\DeclareMathOperator{\hs}{\mathcal{H}(\Gamma)}
\DeclareMathOperator{\lp}{\Delta}
\newcommand{\ra}{\rightarrow}
\newcommand{\cgs}{\mathcal{C}}
\newcommand{\pgs}{\mathcal{P}}

\begin{document}

\title{An effective bound for the Huber constant for cofinite Fuchsian groups}
\author{J.S. Friedman and J. Jorgenson and J. Kramer}

\begin{abstract}
Let $\Gamma$ be a cofinite Fuchsian group acting on hyperbolic two-space $\HH.$ Let $M=\Gamma \setminus \HH $ be  the corresponding quotient space. For $\gamma,$ a closed geodesic of $M$, let $l(\gamma)$ denote its length. The prime geodesic counting function $\pi_{M}(u)$ is defined as the number of $\Gamma$-inconjugate, primitive, closed geodesics $\gamma $  such
that $e^{l(\gamma)} \leq u.$ The \emph{prime geodesic theorem} implies:
$$\pi_{M}(u)=\sum_{0 \leq \lambda_{M,j} \leq 1/4} \text{li}(u^{s_{M,j}}) +
O_{M}\left(\frac{u^{3/4}}{\log{u}}\right),
$$
where $0=\lambda_{M,0} < \lambda_{M,1} < \cdots $ are the eigenvalues
of the hyperbolic Laplacian
acting on the space of smooth functions on $M$ and
$s_{M,j} = \frac{1}{2}+\sqrt{\frac{1}{4} - \lambda_{M,j} }. $ Let $C_{M}$ be smallest implied constant so that
$$
\left|\pi_{M}(u)-\sum_{0 \leq \lambda_{M,j} \leq 1/4} \text{li}(u^{s_{M,j}})\right|
 \leq C_{M}\frac{u^{3/4}}{\log{u}} \quad \text{\text{for all} $u > 1.$ }$$
We call the (absolute) constant $C_{M}$ the Huber constant.

The objective of this paper is to give an effectively computable upper bound of $C_{M}$ for an arbitrary cofinite Fuchsian group. As a corollary we estimate the Huber constant for $\PSL(2,\ZZ),$ we obtain $C_{M} \leq 16,\!607,\!349,\!020,\!658 \approx \exp(30.44086643)$.
\end{abstract}

\maketitle

%\tableofcontents{}
\section*{Introduction}

Let $\Gamma$ be a cofinite Fuchsian group\footnote{The second named  author acknowledges support from grants from the NSF and PSC-CUNY.
The third named author acknowledges support from the DFG Graduate School Berlin Mathematical School and from the
 DFG Research Training Group Arithmetic and Geometry.}, and let $M = \Gamma \setminus \HH $ be the corresponding
hyperbolic orbifold. Let $\cgs(M)$ denote the set of closed geodesics of $M,$ and let $\pgs(M)$ denote the set of  prime (or primitive) closed geodesics (see \cite[page 245]{Buser}). For each $\gamma \in \cgs(M)$ there exists a unique prime geodesic $\gamma_{0}$ and a unique exponent $m \geq 1$ so that $\gamma = \gamma_{0}^{m}.$ Let $l(\gamma)$ denote the \emph{length} of $\gamma.$ Associated to $\gamma$ is a unique hyperbolic conjugacy class $\{ P_{\gamma} \}_{\Gamma} $ with norm
$$N_{\gamma} \df N(P_{\gamma}) = e^{l(\gamma)}.  $$

The prime geodesic counting function $\pi_{M}(u)$ is defined to be the number of
$\Gamma$-inconjugate, primitive, hyperbolic elements $\gamma \in \Gamma$ such
that $e^{l_{\gamma}} < u.$

Selberg \cite{Selberg1} and Huber \cite{Huber1,Huber2,Huber3} independently proved the \emph{prime geodesic theorem},
which is the asymptotic formula
$$\pi_{M}(u) \sim \frac{u}{\log{u}}, \quad (u \ra \infty). $$
Later, Huber proved a stronger version of the prime geodesic theorem, with error terms\footnote{Huber's error term was slightly different: $O(u^{3/4} (\log{u})^{-1/2}).$}:
$$\pi_{M}(u)=\sum_{0 \leq \lambda_{M,j} \leq 1/4} \text{li}(u^{s_{M,j}}) +
O_{M}\left(\frac{u^{3/4}}{\log{u}}\right),
$$
where $0=\lambda_{M,0} < \lambda_{M,1} < \cdots $ are the eigenvalues
of the hyperbolic Laplacian
acting on the space of smooth functions on $M$ and
$$s_{M,j} = \frac{1}{2}+\sqrt{\frac{1}{4} - \lambda_{M,j} }. $$
For a proof of this theorem, via the Selberg zeta function, see \cite{Hejhal1,Hejhal2} and \cite{Randol1}. Randol \cite{Randol2} and Sarnak \cite{Sarnak3}  also gave of proof using the Selberg trace formula.
We call the implied constant inherent in $O_{M}(\frac{u^{3/4}}{\log{u}})$ the \emph{Huber constant,} denoted by $C_{M}.$ The Huber constant is an \emph{absolute constant,} not an
asymptotic constant\footnote{If $C_{M}$ was an asymptotic constant, it would equal zero for the Modular group since there are better estimates for
the error term \cite{Iwaniec}.}. The Huber constant is the minimal constant that satisfies
$$
\left| \pi_{M}(u)-\sum_{0 \leq \lambda_{M,j} \leq. 1/4}
\text{li}(u^{s_{M,j}}) \right| \leq
C_{M}\left(\frac{u^{3/4}}{\log{u}}\right) \quad \text{\text{for all} $u >
1.$}
$$

The main goal of this paper is to estimate $C_{M}$ for an \emph{arbitrary} cofinite Fuchsian group.
We determine an effective algorithm by which one can obtain explicit bounds for $C_{M}$ using elementary
geometric and spectral theoretic data associated to $M$.  In particular, we apply the algorithm
in the case of the full modular group $\PSL(2,\ZZ)$ and obtain a precise, numerical bound for the Huber
constant.
%\cite{Huber1,Huber2,Huber3,JorKra4,JorKra5,JorKra6,Randol1,Randol2,Sarnak3,Hejhal1,Hejhal2}

In \cite{JorKraET}, the authors studied the Huber constant in the following situation: Let $\Gamma_{0}$ be a cofinite Fuchsian group, and let $\Gamma$ be a finite-index subgroup of $\Gamma_{0}.$ They showed that
$$C_{M} \leq [\Gamma_{0}:\Gamma]\cdot C_{M_{0}}. $$
Hence if one could estimate this constant for $\Gamma_{0} = \PSL(2,\ZZ),$  one would have an estimate for any congruence subgroup. For more applications of these ideas, see \cite{JorKra4,JorKra5,JorKra6}.

In the articles \cite{JorKra4}, \cite{JorKra5} and \cite{JorKra6}, the authors studied analytic aspects
of Arakelov theory, ultimately developing bounds for special values of Selberg's zeta functions, Green's
functions and Faltings's delta function.  The bounds for these analytic functions involved many explicitly computable analytic and geometric quantity, together with the Huber constant.  As such, specific, effective bounds for the
Huber constant then could be used to make the bounds derived in \cite{JorKra4}, \cite{JorKra5} and \cite{JorKra6} effective.  With this said, the results in the present paper complete the analysis in the aforementioned articles,
thus providing effective, numerically computable, bounds for various analytic quantities which appear in the
Arakelov theory of algebraic curves.

\subsection*{Main result}
Our main result is Theorem~\ref{thmHuberCons} (see \S\ref{secMainResult}), where we determine an upper bound for the Huber constant $C_{M}$ in terms of various well studied invariants of $\Gamma$ and $M$ such as: the number of small eigenvalues $\sml,$ the number of exceptional poles of the scattering matrix $\pol,$ the smallest positive eigenvalue $\lambda_{1},$ the length of the smallest closed geodesic of $M,$ the decomposition of the fundamental domain of $\Gamma$ into cusp sectors, the area of $M,$ and a few other simple, easily computable, invariants\footnote{These invariants are all known or easily estimated for the Modular Group.}. We could give the main result by stating our upper bound for $C_{M}.$ However the result is so complicated that it would take us two pages just to list out the equation. So, we present our results as an algorithm, rather than a formula. To help the reader appreciate the complexity (for the case of a general cofinite Fuchsian group), we state, as a corollary, the simple case of a cocompact, torsion-free Fuchsian group.

In the process of proving Theorem~\ref{thmHuberCons}, we give an upper bound for the implied (absolute) constant of the spectral counting $\mu(r).$ See Theorem~\ref{thmSpecBound}.  In other words, for an \emph{arbitrary cofinite} Fuchsian group, we find an \emph{explicit} constant $C,$ depending on $\Gamma,$ so that
$$\mu(r) \df |\{r_{n}~|~0 \leq r_{n} \leq r \}| + \frac{1}{4\pi}\int_{-r}^{r} \left| \frac{\phi'}{\phi}(\frac{1}{2}+it) \right| ~dt  \leq C \left(r^{2}+\frac{1}{4} \right) \quad (r \geq 0).$$
%:modular intro results

\begin{table}
\caption{Fundamental constants}
\begin{tabular}{|c|c|}
\hline
Constant&
Value\tabularnewline
\hline
\hline
$\Gamma$ &  cofinite Fuchsian group with fundamental domain $\F$
\tabularnewline
\hline
$\tau$ &  number of inequivalent parabolic cusps
\tabularnewline
\hline
$Y$ &  $\F = \F_{Y} \cup \bigcup_{i=0}^{\tau}\F_{i}^{Y}$, a decomposition of $\F$ into cusp sectors $\F_{i}^{Y}$ and compact set $\F_{Y}$
\tabularnewline
\hline
 $\theta_{R}, m_{R}$ & $\tr(R) = 2\cos(\theta_{R}),$ $0 < \theta_{R} < \pi;$  $m_{R}=|C(R)|$
\tabularnewline
\hline
$d$ & For cocompact $\Gamma$\, $d$ is the (hyperbolic) diameter of $\F$ \tabularnewline
\hline
$\sml$ &  number of eigenvalues $\lambda$ of $\lp,$ such that    $\lambda \leq \frac{1}{4}$,
including $\lambda_{0} = 0$
\tabularnewline
\hline
$s_{1}$ & $ \left\{ \begin{array}{cc} \frac{1}{2} + \sqrt[+]{\frac{1}{4} - \lambda_{1}} &  \lambda_{1} \leq \frac{1}{4}   \\ \infty & \lambda_{1} > \frac{1}{4},  \end{array} \right. $ here $\lambda_{1}$ is the smallest non-zero eigenvalue
\tabularnewline
\hline
$\pol$ &  $\sigma_{1} \geq  \dots \geq \sigma_{\pol}$ are all the poles of $\phi(s)$  in the interval $(\frac{1}{2},1]$
\tabularnewline
\hline
$c_{1}$ &  minimum of the lower left hand matrix entry of non-parabolic elements
\tabularnewline
\hline
$c$ & chosen to satisfy $1 < c < N(\Gamma)$ for all hyperbolic $\gamma \in \Gamma$
\tabularnewline
\hline
$\mu$ & $\log{c},$ a positive lower bound for the length of all closed geodesics
\tabularnewline
\hline
\end{tabular}
\end{table}
\vskip .05in
Next, consider the following table of constants, which are valid only for $\Gamma$ cocompact, torsion-free.

%\newpage

\begin{table}
\caption {Algorithm for cocompact, torsion-free case}
\begin{tabular}{|c|c|}
\hline
Constant&
Value\tabularnewline
\hline
\hline
$C_{1}$& $2e-2,$ $e = 2.718...$ \tabularnewline
\hline
$B$& $\frac{2\pi e^{d}}{|\F|}$ \quad ($d$ is the diameter of $\F$) \tabularnewline
\hline
$C$ & $3 \left( \frac{|\F|}{4\pi} + 745B \right)$
\tabularnewline
\hline
$\sml$ & number of small eigenvalues of $\lp$ including $\lambda_{0}$
\tabularnewline
\hline
$C_{10}$ & $8480 \sqrt{\frac{e}{2\pi}}$
\tabularnewline
\hline
$C_{12}$ & $(\sml-1)(1+3C_{1} + \frac{2}{1-s_{1}}(1+C_{1})) + 2C_{1} + 2$
\tabularnewline
\hline
$C_{13}$& $\frac{41}{6} C \cdot C_{10}$
\tabularnewline
\hline
$C_{16}$ & $C_{12}+C_{13} + \frac{3}{2\pi}|\F| C_{10}$
\tabularnewline
\hline
$C_{17}$& $4\sml +4C_{16}$
\tabularnewline
\hline
$C_{18}$& $4\sml +5C_{16}$
\tabularnewline
\hline
$c$ & $1 < c < N(P)$ for all $P \in \Gamma,$ hyperbolic
\tabularnewline
\hline
$\mu$ & $\log{c}$
\tabularnewline
\hline
$C_{19}$ & $C_{18} + \frac{8\sml + 4C_{18}}{1-1/c}$
\tabularnewline
\hline
$C_{20}$ & $C_{19}+ \frac{8\sml + 4C_{18}}{\mu}$
\tabularnewline
\hline
$C_{21}$& $ |c-2| \frac{1}{\log{2}} + |2-\sqrt{c}| \frac{2}{\log{c}} $
\tabularnewline
\hline
$C_{22}$ & $\frac{1}{1-1/\log(2)}$
\tabularnewline
\hline
$C_{u}$ & $C_{21}\sml + C_{20}\frac{c^{3/4}}{\log{c}} + C_{20} + C_{20}C_{22} + \frac{3}{4}C_{20}C_{21}$
\tabularnewline
\hline
\end{tabular}
\end{table}
\begin{thm*}
Let $\Gamma$ be a cocompact, torsion-free Fuchsian group. Then the constant $C_{u}$ is an upper bound for
the Huber constant $C_{M},$ and the constant $C$ is an upper bound for the implied constant of the spectral counting function of $\lp.$
\end{thm*}

%:main:general case

In order to estimate $C_{M}$ for general cofinite Fuchsian group we need the following modified table. Some constants from the cocompact table are redefined.

\begin{table}
\caption{Algorithm for the general cofinite Fuchsian case}
\begin{tabular}{|c|c|}
\hline
Constant&
Value\tabularnewline
\hline
\hline
$C_{1}$& $2e-2$ \tabularnewline
\hline
$B$& $4\pi Y(Y+1) \sum_{j=1}^{\min(2,\tau)} e^{2d_{j}},$ see Lemma~\ref{lemFischer} for more details 
\tabularnewline
\hline
$C_{3}$& $8 \tau$ \tabularnewline
\hline
$C_{4}$& $5 \sum_{\{R\}} \frac{1}{2m_{R} \sin \theta_{R}} \left( \frac{1}{2\theta_{R}} +\frac{1}{2(\pi - \theta_{R})}  \right) $ \tabularnewline
\hline
$C_{5}$& $ 8 \tau +C_{4}+ \frac{|\F|}{4\pi}+745 B$ \tabularnewline
\hline
$C_{6}$& $ C_{5}+\left(
 2|\log c_{1}|+ \frac{\pol}{(\sigma_{\pol}-\frac{1}{2})^{2}}
\right) \frac{\sqrt{5}}{\sqrt{16\pi }}$ \tabularnewline
\hline
$C_{7}$& $\frac{8}{4\pi}\left( 2|\log c_{1}|+ \frac{\pol}{(\sigma_{\pol}-\frac{1}{2})^{2}} \right)$ \tabularnewline
\hline
$C$ & $3C_{6} + C_{7}$
\tabularnewline
\hline
$\sml$ & number of small eigenvalues of $\lp$ including $\lambda_{0}$
\tabularnewline
\hline
$C_{10}$ & $8480 \sqrt{\frac{e}{2\pi}}$
\tabularnewline
\hline
$C_{12}$ & $(\sml-1)(1+3C_{1} + \frac{2}{1-s_{1}}(1+C_{1})) + 2C_{1} + 2$
\tabularnewline
\hline
$C_{13}$& $\frac{41}{6} C \cdot C_{10}$
\tabularnewline
\hline
$C_{14}$& $C_{10}\frac{296\tau}{3\pi}+C_{10} \frac{\tau }{2}+2\tau \log{2}$
\tabularnewline
\hline
$C_{15}$& $\frac{56C_{10}}{3}\left|\sum_{\{R\}} \frac{1}{2m_{R} \sin \theta_{R}}\right|$
\tabularnewline
\hline
$C_{16}$ & $ C_{12}+C_{13}+C_{14}+ C_{15} + \frac{3}{2\pi} |\F|C_{10}$
\tabularnewline
\hline
$C_{17}$& $4\sml +4C_{16}$
\tabularnewline
\hline
$C_{18}$& $4\sml +5C_{16}$
\tabularnewline
\hline
$c$ & $1 < c < N(P)$ for all $P \in \Gamma,$ hyperbolic
\tabularnewline
\hline
$\mu$ & $\log{c}$
\tabularnewline
\hline
$C_{19}$ & $C_{18} + \frac{8\sml + 4C_{18}}{1-1/c}$
\tabularnewline
\hline
$C_{20}$ & $C_{19}+ \frac{8\sml + 4C_{18}}{\mu}$
\tabularnewline
\hline
$C_{21}$& $ |c-2| \frac{1}{\log{2}} + |2-\sqrt{c}| \frac{2}{\log{c}} $
\tabularnewline
\hline
$C_{22}$ & $\frac{1}{1-1/\log(2)}$
\tabularnewline
\hline
$C_{u}$ & $C_{21}\sml + C_{20}\frac{c^{3/4}}{\log{c}} + C_{20} + C_{20}C_{22} + \frac{3}{4}C_{20}C_{21}$
\tabularnewline
\hline
\end{tabular}
\end{table}
\begin{thm*}
Let $\Gamma$ be a cofinite Fuchsian group. Then the constant $C_{u}$ is an upper bound for the
Huber constant $C_{M},$ and the constant $C$ is an upper bound for the implied constant of  the spectral counting function of $\lp.$
\end{thm*}

To show the usefulness of  Theorem~\ref{thmHuberCons} and Theorem~\ref{thmSpecBound}, we estimate their respective constants ($C$ and $C_{M}$) for the modular group $\PSL(2,\ZZ).$ The results are Theorem~\ref{thmSpeCouMod} and Theorem~\ref{thmPSLHub}.

\begin{thm*}
Let $\Gamma = \PSL(2,\ZZ).$ Let $\mu(r)$ be the spectral counting function \eqref{eqSpeCouFun}, and let
$$C = 1,\!682,\!997. $$ Then for $r \geq 0,$
$$\mu(r) \leq C\left( r^{2}+ \frac{1}{4} \right). $$
\end{thm*}

\begin{thm*}
Let $\Gamma = \PSL(2,\ZZ).$ Then $C_{M} \leq 16,\!607,\!349,\!020,\!658.$
\end{thm*}

Using the theory laid out in \cite{Sarnak2}, and the excellent open-source computational program \textsc{pari/gp} \cite{PARI2}, we can explicitly calculate (and list out) the length spectrum for $\PSL(2,\ZZ).$ Our computations suggest that $$C_{M} < 2. $$ So it seems that our result is not sharp.

Why is our result for $\PSL(2,\ZZ)$ so much larger than the experimental result? There are many reasons. Our bound for the spectral counting function is probably one million times too big (when compared to Weyl's law, an asymptotic result). We are mostly interested in how the Huber constant grows as a function of the various fundamental constants. Estimating the spectral counting function required absolute estimates of various infinite sums over primitive hyperbolic classes; and we used trivial bounds for the hyperbolic counting function. The non-trivial absolute bound being the sole objective of this paper. However, using $\PSL(2,\ZZ)$ specific bounds on the spectral counting function, the result could be markedly improved.

\subsection*{Connection to existing literature.}

As stated above, the bound for the Huber constant obtained in the present paper allows
one to complete the analysis in \cite{JorKra4}, \cite{JorKra5} and \cite{JorKra6}, resulting
in explicit and computable bounds for certain analytic quantities in Arakelov theory. We refer 
the interested reader to these articles for the precise statements of the main results, from 
which one can immediately employ the bounds obtained here in order to further strengthen the
analysis presented in \cite{JorKra4}, \cite{JorKra5} and \cite{JorKra6}.  We will discuss further
applicability of the present paper to other computational problems.

In \cite{BCJMB}, the authors used ideas from Arakelov theory and developed 
a strategy to compute Galois representations modulo $\ell$ associated to a fixed modular 
form of arbitrary weight.  Ultimately, the goal of \cite{BCJMB} is to devise an algorithm, which has
complexity that is polynomial in $\ell$.  A summary of the ideas from \cite{BCJMB} is given in 
\cite{B1}, where the purpose is to focus attention to the weight twelve modular form associated
to $\textrm{PSL}(2,{\mathbf Z})$.  The ideas from \cite{BCJMB} have been used to achieve advances
in many other computation problems; see, for example, \cite{Bo07}, \cite{CL}, \cite{Cou09}, and \cite{La06}
for specific results, as well as the survey article \cite{ChL} for more general comments.

As stated in \cite{BCJMB}, a key component of the algorithm was to determine bounds for Arakelov Green's 
functions.  F. Merkl provided sufficient bounds to complete the algorithm; however, the bounds from
\cite{JorKra5} are sharper and, as a result, improve the efficiency of the algorithm from \cite{BCJMB}.
Theorem 4.8 of \cite{JorKra5} summarizes a bound for the Arakelov Green's function in terms of hyperbolic
data, including a special value of the Selberg zeta function.  This special value was studied in 
\cite{JorKra4} and now can be bounded explicitly using the bounds for the Huber constant developed 
in the present paper.  Thus, the bounds obtained for the Huber constant proved here ultimately can be
used in the algorithms constructed in \cite{BCJMB} and subsequent articles.  

Going forward, it may be possible that the algorithms of \cite{BCJMB}, or similarly developed ideas, could
yield further applications, in which case effective bounds for the analytic aspects of Arakelov theory
would play a role in the implementation of the resulting algorithms.  Indeed, there has been recent progress
in expanding Arakelov theory to allow for singular, or non-compact metrics (see \cite{BKK1}, \cite{BKK2},
\cite{Fre1}, and \cite{Ha09}).  As a result, one can anticipate further computational algorithms which
utilize these recent extensions of Arakelov theory, in which case the bounds from 
\cite{JorKra4}, \cite{JorKra5} and \cite{JorKra6} and the present paper may play a role.

\subsection*{Outline}
In brief, the present paper amounts to a careful analysis of the proof of the prime geodesic theorem,
making sure that each step in the proof yields explicitly computable constants.  However, straightforward
the methodology may sound, there are many details which involve considerable analytic difficulties.

In \S\ref{secSunLem} we prove some useful lemmas and develop necessary notation. Our argument consists of many estimates, each with an implicit constant, and our notation helps to keep track of these constants. At the heart of our argument is a variant of Karamata's Tauberian theorem (Lemma~\ref{lemKara} in \S\ref{secKaramata}) that gives an absolute, instead of an asymptotic, result.

In \S\ref{secHeatKer} we study the heat kernel $e^{-\lp t}$ of the Laplace-Beltrami operator for the hyperbolic orbifold $M = \Gamma \setminus \HH.$ Our main tool is the Selberg trace formula. We specifically show that $e^{-\lp t} = O(1/t)$ for $t \in (0,5],$ and find the corresponding explicit $O$-constant. Note that we focus on the interval (0,5] because  small eigenvalues ($\lambda < 1/4$) are treated differently in many applications of the Selberg trace formula. We will estimate the spectral counting function $\mu(x)$ for $\lambda \geq 1/5$ giving us a small margin of error from the 1/4 cutoff.

Next, in \S\ref{secBoundSpectral}, using results from \S\ref{secHeatKer} and our variant of Karamata's Tauberian theorem (Lemma~\ref{lemKara}), we prove Theorem~\ref{thmSpecBound}, which provides an effective
bound for the spectral counting function.  In \S\ref{secMod} we specialize Theorem~\ref{thmSpecBound} to the case of $\Gamma = \PSL(2,\ZZ).$

Finally, in \S\ref{secHuberBound}, using Theorem~\ref{thmSpecBound}, we obtain an upper bound for the Huber constant (Theorem~\ref{thmHuberCons}) and in \S\ref{secHubPSL}, we specialize to the case of $\Gamma = \PSL(2,\ZZ).$

\subsection*{Concluding remarks.}

We conclude the introduction by addressing a number of points which naturally
arise when studying the computations presented in this article.  

First, it is possible to further contract the algorithm for bounding the Huber constant
and obtain a bound which can be related to the geometry of $M$.  For example, the
constant $c_{1}$ is the radius of the largest isometric circle in a certain fundamental
domain for $M$; see page 38 of \cite{Iw97}.  Also, the number of small eigenvalues can
be bounded by the topology of $M$; see \cite{Zo82}.  However, to employ general results
would only further worsen the bound for $C_{M}$ for special important cases.  For example,
for the congruence subgroups, the Selberg $1/4$ conjecture asserts there is only one small
eigenvalue, whereas the bound from \cite{Zo82} implies that the number of small eigenvalues
grows at most like the genus of the surface, which may be true in general cases but most likely
is not true in special cases of particular importance.  

Second, the heat kernel style of test functions which we analyze could be replaced by studying
a different approach to the prime geodesic theorem, namely by finite difference methods applied
to the test functions denoted by $\psi_{k}$; see \cite{Hejhal2}.  In our investigation, we saw
no significant difference when looking ahead to the problem of bounding $C_{M}$, so we simply
selected one approach.  We will leave the problem of studying other ways of proving the prime
geodesic theorem to see if the methodologies lead to improved bounds for the Huber constant.

Finally, it is somewhat striking to see that in the case of the full modular group the bound
we prove is approximately $\exp(30.45)$ whereas the numerical investigations lead one to believe
that the true bound is $\log 2$, which differ by an order of $10^{14}$.  In Fall 2003, the
second named author (J.J.) was giving a talk at the Courant Institute at NYU on preliminary results
associated to \cite{JorKra6}.  (As described in \cite{BCJMB}, it was then that Peter Sarnak 
communicated to the second named author (J.J.) the problem of bounding Arakelov Green's functions.)
During the lecture, when the results of \cite{JorKra4} were described, Sarnak stated that he felt
it would be difficult to obtain a precise, numerical bound for the Huber constant for any  
group, including the full modular group.  With this stated, the present paper successfully addresses
the problem of determining precise, effective bounds for $C_{M}$, and it remains to be seen how, or 
rather if, methods can be developed to refine the bound for the full modular group.

\section{Preliminary lemmas} \label{secSunLem}
In this section we group together certain lemmas that will be used throughout this paper. Our key lemma is a variant of Karamata's Tauberian Theorem.
\subsection{Big `O' notation}
Let $i \in \NN,$ and let $f,g$ be functions with a common domain $D.$ The statement
$$f(x) = \bo_{i}(g(x))
$$
means that there exists a positive constant $C_{i}$ so that
$$|f(x)| \leq C_{i}|g(x)| \quad \forall x \in D. $$ If the domain in question is not clear, we will explicitly state it. We call the constant $C_{i}$ absolute\footnote{As opposed to an asymptotic constant, say $A_{i}$ that would satisfy $|f(x)| \leq A_{i}|g(x)|$ for $x$ sufficiently large.}.

\subsection{Karamata's Tauberian theorem} \label{secKaramata}
In this section we prove a variation Karamata's Tauberian Theorem \cite[Page 189-192]{Widder}. Our version keeps track of the implied constants in the original theorem.

We will apply the lemma below with the parameter $d = 5.$ Small eigenvalues ($\lambda < 1/4$) are treated differently in many applications of the Selberg trace formula. We will estimate the spectral counting function $\mu(x)$ for $\lambda \geq 1/5$ giving us a small margin of error from the 1/4 cutoff.

\begin{lem} \label{lemKara}
Let $\alpha(t)$ be a non-negative, non-decreasing function, continuous on all but a possibly countable subset of $(0,\infty).$ Suppose that $\alpha(0)=0$ and
that the integral
$$f(s) \df \int_{0}^{\infty}e^{-st}~d\alpha(t) $$
converges for all $s > 0.$ Suppose further that for $C>0,~d > 0$
\beq  f(s) \leq \frac{C}{s}, \quad \text{for all}~ 0< s \leq d.
\eeq
Then $$\alpha(t) \leq 3C t \quad \text{for all}~ t \geq \frac{1}{d}. $$
\end{lem}

\begin{proof}
For $x \in [0,1]$ let
$$g(x) = \left\{
\begin{array}{cc}
 0  & 0 \leq x \leq e^{-1}     \\
 \frac{1}{x} & e^{-1} \leq x \leq 1.
\end{array}
    \right.
 $$
Note that $$g(x) \leq 3 \quad (0 \leq x \leq 1). $$ Since $\alpha(t)$ is non-decreasing, we have, for $s > 0,$
$$\int_{0}^{\infty}e^{-st}g(e^{-st})~d\alpha(t) \leq \int_{0}^{\infty}e^{-st}\,3~d\alpha(t). $$
Next, observing that
$$ g(e^{-w}) = \left\{
\begin{array}{cc}
 e^{w}  &  0 \leq w \leq 1   \\
 0  &     1 < w,
\end{array}
    \right. $$
an elementary calculation using the fact that $\alpha(0)=0$ shows that
$$\int_{0}^{\infty}e^{-st}g(e^{-st})~d\alpha(t) =  \int_{0}^{1/s}e^{-st}g(e^{-st})~d\alpha(t) = \int_{0}^{1/s}e^{-st} e^{st}~d\alpha(t) = \alpha(1/s) - \alpha(0) = \alpha(1/s). $$
Hence we have
$$\alpha(1/s) = \int_{0}^{\infty}e^{-st}g(e^{-st})~d\alpha(t) \leq \int_{0}^{\infty}e^{-st}\,3~d\alpha(t) = 3f(s) \leq \frac{3C}{s} \quad (0 < s \leq d). $$
Setting $t =1/s$ implies that
$$\alpha(t) \leq 3C t \quad \left( t \geq \frac{1}{d} \right). $$
\end{proof}
Later on, we will let $\alpha(t)$ be the spectral counting function for the Laplacian $\lp.$

\subsection{Trivial counting lemma} \label{secTriv}
Let $\Gamma$ be a cofinite Fuchsian group. Let $\pi(x)$ denote the number of primitive conjugacy  classes of hyperbolic elements (of $\Gamma$) of norm not exceeding $x$ (the details are given below). In this section (following \cite[pages 47-49]{Fischer}) we give an explicit constant $B$ so that for all $x \geq 1,$ $\pi(x) \leq Bx. $ This result is called the \emph{trivial counting lemma}.\footnote{Note that the hyperbolic (or geodesic) prime number theorem implies that $\pi(x) = o(x)$ as $x \ra \infty.$ Hence this lemma is called \emph{trivial}, the non-trivial being the prime geodesic theorem. As our goal is an absolute constant, the trivial lemma will be used in a \emph{bootstrapping} argument. }

Let $\rho(z,w)$ denote the hyperbolic distance on $\HH.$ For $T$ hyperbolic, define the \emph{norm} of $T,$ $N(T),$ by
$$N(T) = \exp \left( \inf_{z\in \HH} \rho(z,Tz) \right).$$
Note that the  norm is constant within a fixed conjugacy class $\{T\}_{\Gamma}.$ For $x \geq 1$  define $\pi(x)$ by
$$\pi(x) \df \left| \{~ \{T\}_{\Gamma} \in \Gamma ~|~ N(T) \leq x ~\} \right|.$$

Let $\F=\F_{Y}\cap \F_{1}^{Y} \cap \F_{2}^{Y} \cap \dots \cap \F_{\tau}^{Y}$ be a fundamental domain for $\Gamma,$ where each $\F_{j}^{Y}$ is a cusp sector and $\F_{Y}$ is compact. Let $\zeta_{1},\dots,\zeta_{\tau}$ be a complete list of representatives of the $\Gamma$ equivalence classes of cusps, and let $A_{1},\dots,A_{\tau}$ be elements of $\PSL(2,\RR)$ satisfying $A_{j}\zeta_{j}=\infty.$ In addition, assume that $A_{j}\F_{j}^{Y}=[0,1) \times [Y,\infty).$

If $\tau=1,$ then it follows that $A_{1}\F$ has a positive Euclidean distance $\epsilon$ to the real axis ($\epsilon \leq Y$). If $\tau \geq 2,$ set
$$
\left(
\begin{array}{cc}
 a  & b    \\
 c  & d
\end{array}
    \right) = A_{2}A_{1}^{-1}.
$$
Note that $c \neq 0$ since $A_{2}A_{1}^{-1}$ does not fix $\infty;$ in this case set
$\epsilon = \min \{Y,c^{-2}Y^{-1}\}. $

Let $K_{j} \df \{z \in A_{j}\F ~|~\epsilon \leq \I(z) \leq Y+1\}.$ It follows that $K_{j}$ has compact closure, and hence has finite area $|K_{j}|.$ Finally let $d_{j}$ be the hyperbolic diameter of $K_{j}.$

\bl \text{\rm \cite[pages 47-49]{Fischer}}\label{lemFischer}
Let $\Gamma$ be a cofinite, noncompact Fuchsian group. Then for all $x \geq 1$
$$\pi(x) \leq \sum_{j=1}^{\min(2,\tau)} \frac{4\pi}{|K_{j}|}\left( \cosh(\log x + 2d_{j}) - 1  \right) $$
\el

Note that $\pi(x) = 0$ for $x \leq 1.$
\bl \label{lemTrivCount}
Let $$B =4\pi Y(Y+1) \sum_{j=1}^{\min(2,\tau)} e^{2d_{j}}.  $$ Then  $$\pi(x) \leq Bx \quad(x \geq 0).  $$
\el
\pf
Since $\pi(x) = 0$ for $x \leq 1,$ we assume that $x > 1.$
By Lemma~\ref{lemFischer} we must estimate
$$\sum_{j=1}^{\min(2,\tau)} \frac{4\pi}{|K_{j}|}\left( \cosh(\log x + 2d_{j}) - 1  \right). $$
Note that for each $j,$ $[0,1)\times [Y,Y+1] \subset K_{j}.$ Hence $$|K_{j}| \geq \frac{1}{Y(Y+1)}. $$
The lemma now follows by observing that for $y,d > 0,$
$$\cosh(y+2d) - 1 \leq \cosh(y+2d) \leq \frac{1}{2}(e^{y+2d}+e^{y+2d}) = e^{y+2d}, $$ and by setting $y = \log x.$
\epf

\subsection{Selberg trace formula} \label{secSTF}
The main tool for this paper is the Selberg trace formula (see \cite{Hejhal1,Hejhal2}, \cite{Iwaniec}, \cite{Venkov}).

\begin{thm}[Selberg trace formula]
\begin{multline*}
\sum_{n=0}^{\infty} h(r_{n}) + \frac{1}{4\pi} \int_{\RR} h(r) \frac{-\phi'}{\phi}(\frac{1}{2}+ir)~dr =
\frac{|\F|}{4\pi} \int_{\RR} h(r) \, r\tanh(\pi r)~dr  \\
+ \sum_{\{\gamma \}_{\Gamma}} \frac{\log N(\gamma_{0})}{N(\gamma)^{1/2} - N(\gamma)^{-1/2}}g(\log N(\gamma))
+ \sum_{\{R\}_{\Gamma}} \frac{1}{2m_{R} \sin \theta_{R}}\int_{\RR} \frac{e^{-2r \theta_{R} }}{1+e^{-2\pi r}} h(r)~dr
+ \frac{h(0)}{4} \tr(I-\mathfrak{S}(\frac{1}{2})) \\ -\tau g(0) \log 2 - \frac{\tau}{2\pi } \int_{\RR} h(r) \psi(1+ir)~dr.
\end{multline*}
\end{thm}
Here $\lambda_{n}= \frac{1}{4}+r_{n}^{2}$ are the eigenvalues of the hyperbolic Laplacian $\lp$ and
$$ r_{n} =
\left\{ \begin{array}{cc}
i\sqrt[+]{\frac{1}{4} - \lambda_{n}} & \text{if}~0 \leq \lambda_{n} \leq \frac{1}{4}    \\
\sqrt[+]{\lambda_{n} - \frac{1}{4}  }  &  \text{if}~ \lambda_{n} \geq \frac{1}{4} \end{array} \right.  $$
(see \S\ref{secHeatKer} for some background material on $\lp$); $h(z)$ is an even holomorphic function in the strip $| \I{z}| \leq \frac{1}{2} + \epsilon$ satisfying
$$|h(z)| = O( (|z|+1)^{-2-\epsilon});  $$
$\{ \gamma \}_{\Gamma}$ are the conjugacy classes for the hyperbolic elements of $\Gamma$ with primitive generator $\gamma_{0};$ $\phi(s)$ is the determinant of the scattering matrix $\mathfrak{S}(s)$; $|\F|$ is the hyperbolic area of the fundamental domain $\F$ of  $\Gamma;$ $g(r)$ is the Fourier transform of $h(z);$ $\{ R \}_{\Gamma}$ are conjugacy classes for the elliptic  elements of $\Gamma$ with $m_{R}=|C(R)|$ (the order of the centralizer of $R$ in $\Gamma$) and $\tr(R) = 2\cos(\theta_{R}),$ $0 < \theta_{R} < \pi;$ $\tau$ is the number of cusps of $\Gamma;$ and $\psi(s)$ is the logarithmic derivative of the Gamma function. See see \cite{Hejhal1,Hejhal2}, \cite{Iwaniec}, \cite{Venkov},  for more details.

\section{Heat kernel for the hyperbolic Laplacian} \label{secHeatKer}
Let $\Gamma$ be a cofinite Fuchsian group, and let $\hs$ denote the Hilbert space of measurable functions $f:\HH \mapsto \CC $ satisfying
\begin{itemize}
\item $f(\gamma z) = f(z)$ for all $z \in \HH,~\gamma \in \Gamma$
\item $\int_{\F} |f(z)|~dv(z) < \infty.$
\end{itemize}
Let $\lp:\hs \mapsto \hs$ be the self-adjoint extension of the (two-dimensional hyperbolic) Laplacian.
In this section we will estimate (for $0 < t \leq 5$)
$$\sum_{n=0}^{\infty} e^{-t r_{n}^{2}} + \frac{1}{4\pi} \int_{\RR} e^{-t r^{2}} \frac{-\phi'}{\phi}(\frac{1}{2}+ir)~dr.$$

For $t > 0,$ let
$$h(r) = e^{-t r^{2}}, $$
$$g(x) = \frac{1}{\sqrt{4\pi t}} e^{-x^{2}/(4t)}. $$
Plugging in the pair $h,g$ into  the Selberg trace formula gives an explicit formula for the (regularized) trace of the heat kernel $e^{-\lp t}:$

\begin{multline} \label{eqHeatKer}
\sum_{n=0}^{\infty} e^{-t r_{n}^{2}} + \frac{1}{4\pi} \int_{\RR} e^{-t r^{2}} \frac{-\phi'}{\phi}(\frac{1}{2}+ir)~dr = \\
\frac{|\F|}{4\pi} \int_{\RR} e^{-t r^{2}} \, r\tanh(\pi r)~dr
+ \frac{1}{\sqrt{4\pi t}} \sum_{\{P \}_{\Gamma}} \frac{\log N(P_{0})}{N(P)^{1/2} - N(P)^{-1/2}} e^{-(\log N(P))^{2}/(4t)} \\
+ \sum_{\{R\}_{\Gamma}} \frac{1}{2m_{R} \sin \theta_{R}}\int_{\RR} \frac{e^{-2r \theta_{R} }}{1+e^{-2\pi r}} e^{-t r^{2}}~dr
+ \frac{1}{4} \tr(I-\mathfrak{S}(\frac{1}{2})) \\ -\frac{\tau}{\sqrt{4\pi t}} \log 2 - \frac{\tau}{2\pi } \int_{\RR} e^{-t r^{2}} \psi(1+ir)~dr,
\end{multline}
where $\lambda_{n} = \frac{1}{4} + r_{n}^{2}$ are the eigenvalues of $\lp.$

We will now estimate each of the right-hand-side terms of  equation (\ref{eqHeatKer}) for $0 < t \leq 5$.

\subsection*{Parabolic term}
\bl \label{lemParBound}
Let $C_{3}= 8\tau.$ Then
$$-\frac{\tau \log 2}{\sqrt{4\pi t}} - \frac{\tau}{2\pi}\int_{\RR} e^{-t r^{2}} \psi(1+ir)~dr+\frac{1}{4}\tr(I-\mathfrak{S}(\frac{1}{2})) = \bo_{3}(\frac{1}{t}) \quad (0 < t \leq 5). $$
\el
\pf
We have \cite[page 943 8.361]{Grad}
$$\psi(z)=\frac{\Gamma'}{\Gamma}(z) = \log z + \int_{0}^{\infty} e^{-uz}\left(\frac{1}{u} - \frac{1}{1-e^{-u}}\right)~du \quad (\R(z) > 0). $$
Thus it follows that for $r \in \RR,$
\begin{multline*}
|\psi(1+ir)| \leq |\log (1+ir)|  + \int_{0}^{\infty} \left| e^{-u(1+ir)}\left(\frac{1}{u} - \frac{1}{1-e^{-u}}\right) \right|~du = \\   |\log (1+ir)|  - \int_{0}^{\infty}  e^{-u}\left(\frac{1}{u} - \frac{1}{1-e^{-u}}\right) ~du = |\log (1+ir)| - (\psi(1) - \log 1) = |\log (1+ir)| - \gamma,
\end{multline*} where $\gamma = -0.577...$ is Euler's constant. Hence $$|\psi(1+ir)-\log(1+ir)| \leq -\gamma. $$
Now an elementary estimate shows that
\beq \label{eqDiGammaBound}
|\psi(1+ir)| \leq 4|r|^{1/4} + 4 \quad (r \in \RR),
\eeq
and it follows that
$$\int_{\RR} e^{-t r^{2}}(4|r|^{1/4} + 4) ~dr \leq \frac{27}{t} \quad (0<t \leq 5).
$$
Noting that
$$\frac{\log 2}{\sqrt{4\pi t}} \leq \frac{1/2}{t} \quad (0 <t \leq 5),
$$
and that $\mathfrak{S}(1/2)$ is a unitary ($\tau \times \tau$)  matrix with real entries (hence the diagonal entries are $\pm 1$), concludes the proof.
\epf

\subsection*{Elliptic term}
Next we estimate the elliptic term $$\sum_{\{R\}} \frac{1}{2m_{R} \sin \theta_{R}}\int_{\RR} \frac{e^{-2r \theta_{R} }}{1+e^{-2\pi r}} e^{-tr^{2}}~dr, $$ where $m_{R}$ are integers and $0 < \theta_{R} < \pi.$
\bl \label{lemEllBound}
$$\sum_{\{R\}} \frac{1}{2m_{R} \sin \theta_{R}}\int_{\RR} \frac{e^{-2r \theta_{R} }}{1+e^{-2\pi r}} e^{-tr^{2}}~dr \leq
\sum_{\{R\}} \frac{1}{2m_{R} \sin \theta_{R}} \left( \frac{1}{2\theta_{R}} + \frac{1}{2(\pi - \theta_{R})}  \right).  $$
In particular, if
$$C_{4} = 5 \sum_{\{R\}} \frac{1}{2m_{R} \sin \theta_{R}} \left( \frac{1}{2\theta_{R}} + \frac{1}{2(\pi - \theta_{R})}  \right).  $$ Then
$$\sum_{\{R\}} \frac{1}{2m_{R} \sin \theta_{R}}\int_{\RR} \frac{e^{-2r \theta_{R} }}{1+e^{-2\pi r}} e^{-tr^{2}}~dr = \bo_{4}(\frac{1}{t}) \quad (0 < t \leq 5). $$
\el
\pf
For $0 < \theta < \pi,$ and $t > 0,$
$$
\frac{e^{-2r \theta }}{1+e^{-2\pi r}} e^{-tr^{2}} \leq f_{\theta}(r) \df \left\{
\begin{array}{cc}
e^{-2r\theta}   & r \geq 0     \\
e^{2r(\pi-\theta)}  & r < 0.
\end{array}
    \right.
$$
The lemma now follows from the equation
$$\int_{\RR} f_{\theta_{R}}(r)~dr = \frac{1}{2\theta_{R}} + \frac{1}{2(\pi - \theta_{R})}.$$
\epf

\subsection*{Identity term}
\bl \label{lemIdBound} For all $t > 0$
\beq
 \left| \, \frac{|\F|}{4\pi}\int_{\RR} r e^{-r^{2}t} \tanh(\pi r)~dr  \, \right| \leq \frac{ |\F|}{4\pi}\frac{1}{t}.
\eeq
\el
\pf
Since $|\tanh(\pi r)|\leq 1$ for $r \in \RR,$
\begin{multline}
\frac{|\F|}{4\pi}\int_{\RR} r e^{-r^{2}t} \tanh(\pi r)~dr = \\ 2\frac{ |\F|}{4\pi}\int_{0}^{\infty} r e^{-r^{2}t} \tanh(\pi r)~dr \leq 2\frac{ |\F| }{4\pi}\int_{0}^{\infty} r e^{-r^{2}t} ~dr =\frac{ |\F|}{4\pi}\frac{1}{t}.
\end{multline}
\epf

\subsection*{Hyperbolic term}
Since $\Gamma$ is discrete, there exists a hyperbolic conjugacy class $\{ P_{s} \}$   with minimal norm $N(P_{s}) > 1.$ Choose $c > 0$ so that $$1 < c < N(P_{s}). $$

\bl \label{lemHypBound}
For  $~0< t \leq 5,$
$$\frac{1}{\sqrt{4\pi t}} \sum_{\{P \}_{\Gamma}} \frac{\log{N(P_{0})}}{N(P)^{1/2}-N(P)^{-1/2} } e^{-(\log{N(P)})^{2}/(4t)} \leq \frac{745B}{t}. $$
\el
\pf
First note that for a hyperbolic element $P,$
$$N(P) = N(P_{0})^{j} $$ for some integer $j \geq 0.$ Since $N(P) > 1,$ $N(P_{0}) \leq N(P). $ Noting that
$$\frac{\log x}{x^{1/2}-x^{-1/2}} \leq 1 \quad (x > 1),$$ implies
\begin{multline*}
\frac{1}{\sqrt{4\pi t}} \sum_{\{P \}_{\Gamma}} \frac{\log{N(P_{0})}}{N(P)^{1/2}-N(P)^{-1/2} } e^{-(\log{N(P)})^{2}/(4t)}  \\ \leq  \frac{1}{\sqrt{4\pi t}} \sum_{\{P \}_{\Gamma}} \frac{\log{N(P)}}{N(P)^{1/2}-N(P)^{-1/2} } e^{-(\log{N(P)})^{2}/(4t)} \leq \frac{1}{\sqrt{4\pi t}} \sum_{\{P \}_{\Gamma}}  e^{-(\log{N(P)})^{2}/(4t)}.
\end{multline*}
Rewriting the last sum as a Stieltjes integral
$$
\frac{1}{\sqrt{4\pi t}} \sum_{\{P \}_{\Gamma}}  e^{-(\log{N(P)})^{2}/(4t)} = \frac{1}{\sqrt{4\pi t}} \int_{c}^{\infty} e^{-(\log x)^{2}/(4t)}~d\pi(x).
$$
Integrating by parts, and the trivial counting lemma (Lemma~\ref{lemTrivCount}) yield
\begin{eqnarray*}
\frac{1}{\sqrt{4\pi t}} \int_{c}^{\infty} e^{-(\log x)^{2}/(4t)}~d\pi(x) &\leq&  \frac{1}{\sqrt{4\pi t}} \int_{c}^{\infty} \frac{\log{x}}{2tx} e^{-(\log x)^{2}/(4t)}\,Bx~dx \\
&=& \frac{B}{2t} \frac{1}{\sqrt{4\pi t}}  \int_{c}^{\infty} (\log{x}) e^{-(\log x)^{2}/(4t)}~dx \\
&\leq& \frac{B}{2t} \frac{1}{\sqrt{4\pi t}} \int_{1}^{\infty} (\log{x}) e^{-(\log x)^{2}/(4t)}~dx \\
&=& B \left(\frac{1}{2}e^{t} \, \text{erf}(\sqrt{t}) + \frac{1}{2}e^{t} + \frac{1}{\sqrt{4\pi t}} \right) \\
&\leq&  B \left(e^{t} + \frac{1}{\sqrt{4\pi t}} \right) \\
&\leq& \frac{745 B}{t} \quad (0 < t \leq 5).
\end{eqnarray*}
The last estimate is a \emph{trivial estimate.}
\epf

\subsection*{Spectral terms}
We have now bounded all the terms on the right side of the trace formula. By Lemmas \ref{lemParBound}, \ref{lemEllBound}, \ref{lemIdBound}, and \ref{lemHypBound}; we have
\bl
Let $$C_{5} = 8 \tau +C_{4}+ \frac{|\F|}{4\pi}+745 B.$$ Then
$$ \left| \sum_{n=0}^{\infty} e^{-t r_{n}^{2}} + \frac{1}{4\pi} \int_{\RR} e^{-t r^{2}} \frac{-\phi'}{\phi}(\frac{1}{2}+ir)~dr \right| = \bo_{5}(\frac{1}{t}) \quad(0 < t \leq 5). $$
\el

\section{Explicit bound for the spectral counting function} \label{secBoundSpectral}
We would like to apply Lemma~\ref{lemKara} to the equation
\beq \label{eqTemp83200}
\left| \sum_{n=0}^{\infty} e^{-t r_{n}^{2}} + \frac{1}{4\pi} \int_{\RR} e^{-t r^{2}} \frac{-\phi'}{\phi}(\frac{1}{2}+ir)~dr \right| = \bo_{5}(\frac{1}{t}) \quad(0 < t \leq 5) \eeq
but a minor problem stops us: The real valued function $\frac{-\phi'}{\phi}(\frac{1}{2}+ir)$ may be negative for some values of $r.$ However, at the cost of a small error term, we will (following Selberg) replace the signed function $\frac{-\phi'}{\phi}(\frac{1}{2}+ir)$ with a non-negative function.

\subsection*{Notation}
Recall from \S\ref{secTriv} that $\zeta_{1},\dots,\zeta_{\tau}$ are a complete list of representatives of the $\Gamma$ equivalence classes of cusps, and  $A_{1},\dots,A_{\tau}$ are elements of $\PSL(2,\RR)$ satisfying $A_{j}\zeta_{j}=\infty.$ Set
$$\Gamma^{i} \df A_{i} \Gamma A_{i}^{-1}. $$ It follows that $\infty$ is a cusp for $\Gamma^{i}.$

For each $\gamma \in \PSL(2,\RR)$ let $c(\gamma)$ denote the lower left-hand matrix entry of $\gamma.$
Now set $$\Delta_{i} \df \{\gamma \in \Gamma^{i}~|~\gamma(\infty) \neq \infty \},$$
and let
\beq \label{eqc1}
c_{1} = \min_{i=1\dots \tau} \{|c(\gamma)|~|~\gamma \in \Delta_{i} \}.
\eeq

Since $\Gamma$ is discrete, Shimizu's Lemma \cite[page 48]{Elstrodt} asserts\footnote{For $\PSL(2,\ZZ),$ $c_{1}=1.$} that $c_{1} > 0.$

\subsection*{Modified scattering matrix}
Recall from \S\ref{secSTF} that $\phi(s)$ is the determinant of the \emph{scattering matrix.} Let $\sigma_{1} \geq \sigma_{2} \geq \dots \geq \sigma_{\pol}$ be all the poles of $\phi(s)$ in the interval $(\frac{1}{2},1],$ and let $$\phi_{r}(s) \df c_{1}^{2s-1}\phi(s) \prod_{j=1}^{\pol} \frac{s-\sigma_{j}}{s-1+\sigma_{j}}. $$

\bl  \label{lemModSca} For all $r \in \RR,$
\begin{enumerate}
\item $\frac{-\phi_{r}'}{\phi_{r}}(\frac{1}{2}+ir) $ is non-negative.
\item
$$\left| \frac{-\phi'}{\phi}(\frac{1}{2}+ir) -\frac{-\phi_{r}'}{\phi_{r}}(\frac{1}{2}+ir) \right| \leq 2|\log c_{1}| + \frac{\pol}{(\sigma_{\pol}-\frac{1}{2})^{2} + r^{2} }.  $$
\end{enumerate}
\el
\pf
Item (1) is proved in \cite[Theorem 3.5.5]{Venkov}. (2) follows from elementary estimates: Taking the logarithmic derivative of both sides of the equation below
$$\theta_{r}(s) = c_{1}^{2s-1}\phi(s) \prod_{j=1}^{\pol} \frac{s-\sigma_{j}}{s-1+\sigma_{j}},$$
plugging in $s=\frac{1}{2}+ir,$ and recalling that $ \frac{1}{2} < \sigma_{\pol} \leq \dots \leq \sigma_{1} \leq 1;$  we obtain
\begin{multline*}
  \left| \frac{-\phi'}{\phi}(\frac{1}{2}+ir) -\frac{-\phi_{r}'}{\phi_{r}}(\frac{1}{2}+ir)  \right| =\left| -2\log c_{1} + \sum_{j=1}^{\pol} \frac{1}{(-\frac{1}{2}+ir - \sigma_{j})(\frac{1}{2}+ir - \sigma_{j}) } \right| \leq \\ |2\log c_{1}| + \sum_{j=1}^{\pol} \frac{1}{(\sigma_{\pol}-\frac{1}{2})^{2} + r^{2} } = 2|\log c_{1}| + \frac{\pol}{(\sigma_{\pol}-\frac{1}{2})^{2} + r^{2} }.
\end{multline*}
\epf

Let
$$W(\lambda) = \left\{
\begin{array}{cc}
 \frac{1}{4\pi}\int_{-\sqrt{\lambda - \frac{1}{4}}}^{\sqrt{\lambda - \frac{1}{4}}} \left|\frac{\phi'}{\phi}(\frac{1}{2}+ir)\right|~dr,  & \lambda \geq \frac{1}{4},     \\
0,  &  0 \leq \lambda \leq \frac{1}{4},
\end{array}
    \right. $$
$$Q(\lambda) = \left\{
\begin{array}{cc}
 \frac{1}{4\pi}\int_{-\sqrt{\lambda - \frac{1}{4}}}^{\sqrt{\lambda - \frac{1}{4}}}\frac{-\phi_{r}'}{\phi_{r}}(\frac{1}{2}+ir)~dr,  & \lambda \geq \frac{1}{4},     \\
0,  &  0 \leq \lambda \leq \frac{1}{4},
\end{array}
    \right. $$
and let $ N(\lambda) = | \{ \lambda_{n} ~|~ \lambda_{n} \leq \lambda  \} |.$

%:Main Spec Lemma

\bl \label{lemSpeMain}
Let $$C_{6} = C_{5}+\left(
 2|\log c_{1}|+ \frac{\pol}{(\sigma_{\pol}-\frac{1}{2})^{2}}
\right) \frac{\sqrt{5}}{\sqrt{16\pi }},$$

$$C_{7}=\frac{8}{4\pi}\left( 2|\log c_{1}|+ \frac{\pol}{(\sigma_{\pol}-\frac{1}{2})^{2}} \right). $$
Then the following hold:
\begin{enumerate}
\item $$\int_{0}^{\infty}e^{-t \lambda}d(N(\lambda)+Q(\lambda)) \leq   \frac{C_{6}}{t} \quad (0 < t \leq 5). $$
\item $N(\lambda)+Q(\lambda) \leq 3C_{6}\lambda \quad (\lambda \geq \frac{1}{4}). $
\item $N(\lambda) + W(\lambda) \leq (3C_{6}+C_{7})\lambda \quad (\lambda \geq \frac{1}{4}). $
\end{enumerate}
\el
\pf
(1) A simple calculation using $\lambda_{n} = \frac{1}{4} + r_{n}^{2},$ shows that
$$\int_{0}^{\infty}e^{-t \lambda}d(N(\lambda)+Q(\lambda)) = e^{-t/4} \left( \sum_{n=0}^{\infty} e^{-t r_{n}^{2}} + \frac{1}{4\pi} \int_{\RR} e^{-t r^{2}} \frac{-\phi_{r}'}{\phi_{r}}(\frac{1}{2}+ir)~dr \right). $$

Since $e^{-t/4} \leq 1,$
$$
\int_{0}^{\infty}e^{-t \lambda}d(N(\lambda)+Q(\lambda)) \leq  \left( \sum_{n=0}^{\infty} e^{-t r_{n}^{2}} + \frac{1}{4\pi} \int_{\RR} e^{-t r^{2}} \frac{-\phi_{r}'}{\phi_{r}}(\frac{1}{2}+ir)~dr \right).$$
Note that both sides of the above inequality are positive. Let
\beq \label{eqTempl83js}
U(r)=\frac{-\phi'}{\phi}(\frac{1}{2}+ir) - \frac{-\phi_{r}'}{\phi_{r}}(\frac{1}{2}+ir),
\eeq
Then rewriting \eqref{eqTemp83200} gives us
\beq \label{eqTemp9sdfj}
\left| \sum_{n=0}^{\infty} e^{-t r_{n}^{2}} + \frac{1}{4\pi} \int_{\RR} e^{-t r^{2}} \left( \frac{-\phi_{r}'}{\phi_{r}}(\frac{1}{2}+ir) + U(r) \right)~dr \right| \leq \frac{C_{5}}{t}   \quad (0 < t \leq 5).
\eeq
But Lemma~\ref{lemModSca}, and a crude estimate (see equation (\ref{eqUtriv}) for the idea of the estimate), show that
$$\frac{1}{4\pi }\int_{\RR} e^{-t r^{2}} |U(r)| ~dr \leq
 \left(
 2|\log c_{1}|+ \frac{\pol}{(\sigma_{\pol}-\frac{1}{2})^{2}}
\right)   \frac{1}{\sqrt{16\pi }} \frac{1}{\sqrt{t}},$$
and since $$ \frac{1}{\sqrt{t}} \leq \frac{\sqrt{5}}{t} \quad (0 < t \leq 5),$$
it follows that
\beq
\left| \sum_{n=0}^{\infty} e^{-t r_{n}^{2}} + \frac{1}{4\pi} \int_{\RR} e^{-t r^{2}}  \frac{-\phi_{r}'}{\phi_{r}}(\frac{1}{2}+ir) ~dr \right| \leq \frac{C_{5}+\left(
 2|\log c_{1}|+ \frac{\pol}{(\sigma_{\pol}-\frac{1}{2})^{2}}
\right) \frac{\sqrt{5}}{\sqrt{16\pi }}}{t}  \quad (0 < t \leq 5).
\eeq
(1) is now proved.

(2) follows from that fact that $N(\lambda) + Q(\lambda)$ is non-decreasing, (1), and Lemma~\ref{lemKara}.

To prove (3), note that  \eqref{eqTempl83js} implies that
$$N(\lambda) + W(\lambda) \leq N(\lambda) + Q(\lambda) + \frac{1}{4\pi} \int_{-\sqrt{\lambda - \frac{1}{4}}}^{\sqrt{\lambda - \frac{1}{4}}} |U(t)|~dt.  $$
Next, using Lemma~\ref{lemModSca}, a trivial estimate gives
\begin{multline} \label{eqUtriv}
\int_{-\sqrt{\lambda - \frac{1}{4}}}^{\sqrt{\lambda - \frac{1}{4}}} |U(t)|~dt \leq \int_{-\sqrt{\lambda - \frac{1}{4}}}^{\sqrt{\lambda - \frac{1}{4}}} \left( 2|\log c_{1}| + \frac{\pol}{(\sigma_{\pol}-\frac{1}{2})^{2} + r^{2} }  \right)  ~dt \\ \leq  \int_{-\sqrt{\lambda - \frac{1}{4}}}^{\sqrt{\lambda - \frac{1}{4}}} \left( 2|\log c_{1}| + \frac{\pol}{(\sigma_{\pol}-\frac{1}{2})^{2}}  \right)  ~dt.
\end{multline}
(3) now follows from the equation
$$\sqrt{\lambda+\frac{1}{4}} \leq 4 \lambda \quad (\lambda \geq 1/4 ). $$
\epf

\subsection{Implied constant for the spectral counting function} \label{secConSpeCou}
Recall that $\lambda_{n}= \frac{1}{4}+r_{n}^{2}.$ Also  note that for $\lambda_{n} \geq \frac{1}{4}, $ we have $r_{n} \geq 0.$

Let
\beq \label{eqSpeCouFun}
\mu(r) = |\{r_{n}~|~0 \leq r_{n} \leq r \}| + \frac{1}{4\pi}\int_{-r}^{r} \left| \frac{\phi'}{\phi}(\frac{1}{2}+it) \right| ~dt.
\eeq
Lemma~\ref{lemSpeMain}(3) immediately implies that
$$\mu(r)  \leq  N(r^{2}+\frac{1}{4}) + W(r^{2}+\frac{1}{4}) \leq (3C_{6}+C_{7}) (r^{2}+\frac{1}{4}) \quad (r \geq 0).$$ We summarize our result:

\begin{thm} \label{thmSpecBound} Let $C=3C_{6}+C_{7}.$
Then
$$\mu(r) \leq C \left(r^{2}+\frac{1}{4} \right) \quad (r \geq 0).$$
\end{thm}

\subsection{Explicit estimate for the modular group} \label{secMod}
Let $\Gamma = \PSL(2,\ZZ).$ In this section we will give an explicit upper bound for the number $C$ which depends on $B,C_{3},C_{4}, C_{5}, C_{6},$ and $C_{7}.$

\textbf{The constant $B.$ } From the trivial counting lemma (\S\ref{secTriv}) we have
$$B =4\pi Y(Y+1) \sum_{j=1}^{\min(2,\tau)} e^{2d_{j}}, $$
where $\tau = 1.$  Let $\F$ be the standard fundamental domain of $\Gamma$
$$\F = \{z \in \HH~|~ -\frac{1}{2} \leq \R{z} \leq \frac{1}{2},~|z| \geq 1 \}. $$ The euclidian distance of $\F,$ to the $x$-axis,  is $\epsilon = \frac{\sqrt{3}}{2}.$
We can decompose $\F=\F_{0} \cup \F(Y),$ where
$$\F(Y) = \{z \in \HH~|~ -\frac{1}{2} \leq \R{z} \leq \frac{1}{2},~\I{z} > Y \},$$
with $Y=2.$

Recall that  $K_{1} = \{z \in \F ~|~\frac{3}{2} \leq \I(z) \leq Y+1\}.$  Let $d_{1}$ be the hyperbolic diameter of $K_{1}.$ To estimate $d_{1}$ note that
$\cosh(\rho(z,w)) = 1 + \frac{|z-w|^{2}}{\I{z} \I{w}}, $
where $\rho(z,w)$ is the hyperbolic distance between $z$ and $w.$ An elementary calculation shows that $d_{1} < 2.3.$
Hence $B < 753.$

\textbf{The constant $C_{4}.$} It is known that $\Gamma$ has two classes of elliptic elements, represented by
$$
R = \left( \begin{array}{cc}
0 & -1   \\
1  & 0   \end{array} \right),
$$
and
$$
S = \left( \begin{array}{cc}
1 & -1   \\
1  & 0   \end{array} \right).
$$
Elementary calculations show that $m_{R} = 2$ (the order of the centralizer of $R$ in $\Gamma$) and that $0 = 2\cos(\theta_{R})$ where $0 < \theta_{R} < \pi;$ thus $ \theta_{R} = \frac{\pi}{2}.$ Similarly $m_{S}=3$ and $\theta_{S} = \frac{\pi}{3}.$ From Lemma~\ref{lemEllBound} it follows that $C_{4} < \frac{3}{2}.$

\textbf{The remaining constants.} Now $C_{3} = 8\tau =8$, and $|\F| = \frac{\pi}{3}.$

The number $c_{1} = 1$ and there is only one pole (at $s=1$) for the scattering matrix $\phi(s)$, so $\pol = 1$ and $\sigma_{\pol} = 1.$ Now $C_{6},$ $C_{7},$ and $C$ can be easily computed.
The result is

\begin{thm} \label{thmSpeCouMod}
Let $\Gamma = \PSL(2,\ZZ).$ Let $\mu(r)$ be the spectral counting function \eqref{eqSpeCouFun}, and let
$$C = 1,\!682,\!997. $$ Then for $r \geq 0,$
$$\mu(r) \leq C\left( r^{2}+ \frac{1}{4} \right). $$
\end{thm}

The estimate for $C$ is probably not sharp. However, we did not use any of the number theoretic properties of $\PSL(2,\ZZ)$ to arrive at the estimate; we only assumed that $\Gamma$ was a cofinite Fuchsian group.

\section{Bound for the Huber constant} \label{secHuberBound}
In this section, we adapt Randol's proof  (\cite{Randol2}) of the prime geodesic theorem to an arbitrary cofinite Fuchsian group. Along the way, using Theorem~\ref{thmSpecBound}, we will obtain an upper bound for the Huber constant $C_{M}.$ Our main references are \cite[pages 295-300]{Randol2} and \cite[Section 9.6]{Buser}.

\subsection{Explicit bump function} \label{secBump}
For $T > 0$ set
$$I_{T}(x)= \left\{
\begin{array}{cc}
 1  & |x| \leq T,   \\
 0  & |x| > T.
\end{array}
    \right. $$

For a functions $f,g$ on $\RR,$ we set
$$\widehat{f}(x) \df  \frac{1}{\sqrt{2\pi}} \int_{\RR} f(y)e^{-ixy}~dy $$
and
$$ (f*g)(x) \df \frac{1}{\sqrt{2\pi}} \int_{\RR} f(x-y)g(y)~dy. $$

\bl \label{lemBump} There exists a nonnegative function $\phi(x)$ with the following properties:
\begin{enumerate}
\item \label{itPhiA} $\phi \in C_{0}^{\infty}(\RR)$ with $\supp(\phi) \subset [-1,1]$.
\item  \label{itPhiB} $\widehat{\phi}(0) = 1$.
\item \label{itPhiC} Let $C_{1}=2(e-1). $ Then $\widehat{\phi}(\epsilon z) -1 = \bo_{1}(\epsilon)$  for $0 < \epsilon \leq 1,$ and $z \in i[-\frac{1}{2},1]$.
\item \label{itPhiD} Let $b>0,$ $C_{2}^{(b)}=\frac{848}{\sqrt{2\pi}}e^{b}.$ Then for $z=r+it, |t|\leq b,$ we have
$$
\widehat{\phi}(z) = \bo_{2}( (1+|r|)^{-2} ),  \quad (r\in \RR, |t| \leq b).
$$
\end{enumerate}
\el
\pf
Let $$c_{0} = \int_{-1}^{1}\exp\left(\frac{1}{x^{2}-1}\right)~dx = 0.4439938\dots,  $$
and let
$$\phi(x) = \frac{\sqrt{2\pi}}{c_{0}}\left\{
\begin{array}{cc}
 \exp({\frac{1}{x^{2} - 1}})  &  |x| < 1,  \\
 0  &	|x| \geq 1.
\end{array}
    \right.
$$
(\ref{itPhiA}) and (\ref{itPhiB}) now follow from the definition of $\phi.$

For $w\in \CC, |w|\leq 1,$ observe that
$$|e^{w}-1| = |w+\frac{w^{2}}{2}+\frac{w^{3}}{3!}+\dots,| = |w|\cdot|1+\frac{|w|}{2}+\frac{|w|^{2}}{3!}+\dots,| \leq |w|\cdot |1+\frac{1}{2}+\frac{1}{3!}+\dots,|= (e-1)|w|.  $$
For $0 < \epsilon \leq 1,$ and $z \in i[-\frac{1}{2},1],$ note that $|\epsilon z| \leq 1.$ Let $w = \epsilon z.$ By definition, and (\ref{itPhiB}), it follows that
\begin{multline*}
\widehat{\phi}(\epsilon z) -1   = \frac{1}{\sqrt{2\pi}} \int_{\RR} \phi(y)e^{-iwy}~dy - 1 = \\ \frac{1}{\sqrt{2\pi}} \int_{\RR} \phi(y)e^{-iwy}~dy - \frac{1}{\sqrt{2\pi}} \int_{\RR} \phi(y)~dy =\frac{1}{\sqrt{2\pi}}\int_{-1}^{1} \phi(y)(e^{-iwy}-1)~dy.
\end{multline*}
Now, since $$\frac{\phi(y)}{\sqrt{2\pi}} \leq 1 \quad (y \in \RR), $$
it follows that
$$|\widehat{\phi}(\epsilon z) -1| \leq  \int_{-1}^{1}|e^{-iwy}-1|~dy \leq  \int_{-1}^{1}(e-1)|w|~dy=2(e-1)|w| = 2(e-1)|\epsilon z| \leq 2(e-1)\epsilon. $$ This proves (\ref{itPhiC}).

Next we prove (\ref{itPhiD}). Suppose $|\I(z)|\leq b.$
Then  since  $\widehat{\phi}(0)=1$ it follows that
$$|\widehat{\phi}(z)| \leq e^{b}.$$
Integrating by parts, twice, shows that for $z \neq 0,$
$$|\widehat{\phi}(z)| \leq \frac{1}{|z|^{2}}\frac{1}{\sqrt{2\pi}} e^{b}\int_{-1}^{1} |\phi''(y)|~dy.  $$
Now, a simple calculation\footnote{First note that $\phi''(x) = \frac{\phi(x)}{(x^{2}-1)^{4}}(3x^{4}-1). $ Next using elementary calculus maximize $\frac{\phi(x)}{(x^{2}-1)^{4}} $ and $(3x^{4}-1),$ separately, and multiply their maximums. Of course, this argument does not yield the optimal bound. } shows that $$|\phi''(y)| \leq 106 \quad (-1 \leq y \leq 1).$$ Hence, for $z \neq 0,$
$$|\widehat{\phi}(z)| \leq \frac{1}{|z|^{2}}\frac{212}{\sqrt{2\pi}}e^{b}. $$
Finally, observing that
$$1 \leq \frac{2}{1+|z|^{2}} \quad (|z| \leq 1), $$ and
$$\frac{1}{|z|^{2}} \leq \frac{2}{1+|z|^{2}} \quad (|z| \geq 1), $$ it follows that
$$|\widehat{\phi}(|z|)| \leq \frac{1}{|z|^{2}+1}\frac{424}{\sqrt{2\pi}}e^{b}\leq \frac{1}{r^{2}+1}\frac{424}{\sqrt{2\pi}}e^{b} \leq  \frac{1}{(|r|+1)^{2}}\frac{848}{\sqrt{2\pi}}e^{b}$$
since $$\frac{1}{r^{2}+1} \leq \frac{2}{(|r|+1)^{2}} \quad (r \neq 0).$$ This concludes the proof.
\epf

\subsection{Explicit lemmas}
In this section we prove some lemmas which allow us to obtain explicit constants from Randol's proof of the prime geodesic theorem.
Recall (\S\ref{secBump}), and the definitions of
$I_{T}(x)$ and $\phi.$ For $\epsilon > 0,$ define
$$\phi_{\epsilon}(x) \df \epsilon^{-1}\phi(\epsilon^{-1}x),$$
and define
$$g_{T}^{\epsilon}(x) \df \left( 2\cosh(\frac{x}{2})\right)(I_{T}*\phi_{\epsilon})(x).$$

Set
$$h_{T}^{\epsilon}(r)=\widehat{g_{T}^{\epsilon}}(r). $$
It follows that
$$h_{T}^{\epsilon}(r)=S(r+\frac{i}{2})+S(r-\frac{i}{2}),   $$
where $$S(w)=\left(\frac{2}{w} \sin{(Tw)} \right)\widehat{\phi_{\epsilon}}(w), \quad S(0)=2T. $$ Note that
$$\widehat{\phi_{\epsilon}}(w) = \widehat{\phi}(\epsilon w).$$

The main idea behind the proof is to apply the Selberg trace formula to the pair $h_{T}^{\epsilon}(r),~g_{T}^{\epsilon}(x);$ and to extract the leading and error term from the trace formula. Much of our effort will be spent on giving explicit estimates to each term in the trace formula.

We next set up some important notation.
For $0 \leq \lambda_{k} \leq \frac{1}{4},$ a \emph{small eigenvalue,} define $r_{k},$ and $s_{k},$ by
$$r_{k}\df i\sqrt[+]{\frac{1}{4}-\lambda_{k}} \df i(s_{k}-\frac{1}{2}).$$

Let $\sml$ denote the number small eigenvalues:
$$ \sml \df | \{\lambda_{k}~|~\lambda_{k} \leq  \frac{1}{4}   \} |. $$
Explicitly, $0 = \lambda_{0} < \lambda_{1} \leq \cdots \leq \lambda_{\sml} \leq \frac{1}{4}. $

\bl \label{lemHest}
Let $C_{10}=10\frac{848}{\sqrt{2\pi}}\sqrt{e},$ then for  $r \geq 0, ~0 < \epsilon \leq 1, $
\begin{enumerate}
\item $|h_{T}^{\epsilon}(r)| \leq C_{10}e^{T/2}(1+r)^{-1}(1+\epsilon r)^{-2}. $
\item For $r\in [0,\frac{1}{\epsilon}],$ $$|h_{T}^{\epsilon}(r)| \leq C_{10}\frac{e^{T/2}}{1+r} $$
\item For $r\in [\frac{1}{\epsilon},\infty),$
 $$|h_{T}^{\epsilon}(r)| \leq C_{10}\frac{e^{T/2}}{\epsilon^{2}r^{3}}. $$

\end{enumerate}
\el
\pf
We prove (1);  (2) and (3) will then follow from elementary estimates.

Recall that for $r \geq 0,$
$$h_{T}^{\epsilon}(r)=S(r+\frac{i}{2})+S(r-\frac{i}{2}),   $$
where
$$S(w)=\left(\frac{2}{w} \sin{(Tw)} \right)\widehat{\phi_{\epsilon}}(w),
$$
and note that $\widehat{\phi_{\epsilon}}(w) = \widehat{\phi}(\epsilon w).$
By Part~\ref{itPhiD} of Lemma~\ref{lemBump}, with $b=1/2,$ it follows that
$$\widehat{\phi}(\epsilon(r \pm \frac{i}{2})) =O_{2}((1+\epsilon r)^{-2}),$$ where $$C_{2} \df C_{2}^{(1/2)}=\frac{848}{\sqrt{2\pi}}\sqrt{e}.$$
Now
$$\frac{2}{|r \pm \frac{i}{2}|} \leq \frac{5}{1+r}, $$ and
$$ |\sin{T(r \pm \frac{i}{2})}| \leq e^{T/2}.$$
Hence,
$$|h_{T}^{\epsilon}(r)| \leq 2\frac{5}{1+r}C_{2}(1+\epsilon r)^{-2}.$$
\epf

The next lemma will be used repeatedly.
\bl \label{lemParts}
Let $u(r)$ be a non-decreasing function on $[0,\infty)$ that satisfies $$|u(r)| \leq cr^{2}+b \quad(r \geq 0);~c,b>0. $$
Then
$$\int_{0}^{\infty}|h_{T}^{\epsilon}(r)|~du(r) \leq K \frac{e^{T/2}}{\epsilon} $$
where $K = C_{10}\left(6c+\frac{10b}{3}  \right).$
\el
\pf
By Lemma~\ref{lemHest}, it suffices to estimate
\beq
\int_{0}^{1/\epsilon} \frac{1}{1+r}~du(r) \quad \text{and} \quad
\int_{1/\epsilon}^{\infty} \frac{1}{r^{3}}~du(r).
\eeq
Integrating by parts, yields
\begin{multline*}
\int_{0}^{1/\epsilon} \frac{1}{1+r}~du(r) = \left.\frac{u(r)}{1+r}\right]_{0}^{1/\epsilon} + \int_{0}^{1/\epsilon}\frac{u(r)}{(1+r)^{2}}~dr \leq \\ \frac{c/\epsilon^{2}+b}{1+1/\epsilon}+b +  \int_{0}^{1/\epsilon}\frac{cr^{2}+b}{(1+r)^{2}}~dr \leq \frac{c}{\epsilon} + b + \frac{3c}{\epsilon} + \frac{b}{\epsilon} \leq \frac{4c+2b}{\epsilon}.
\end{multline*}
A similar argument shows that
$$
\int_{1/\epsilon}^{\infty} \frac{1}{r^{3}}~du(r) \leq (c+b)\epsilon + (c+\frac{b}{3})\epsilon = (2c+\frac{4b}{3})\epsilon.
$$
Hence
$$\int_{0}^{\infty}|h_{T}^{\epsilon}(r)|~du(r) \leq C_{10}\left(6c+\frac{10b}{3}  \right)\frac{e^{T/2}}{\epsilon}. $$
\epf

\subsection{The leading term of the prime geodesic theorem}
In this section we show that the sum
$$\sum_{\lambda_{k} \leq \frac{1}{4}} h_{T}^{\epsilon}(r_{k})$$
gives us the leading term of the prime geodesic theorem. In the next section, we will see that the above sum is a term in the Selberg trace formula.
\bl \label{lemSmallEigBounds}
Let $ \epsilon = e^{-T/4},$ and
$$
C_{12}  \df (\sml-1)\left(1+3C_{1} + \frac{2}{1-s_{1}}(1+C_{1})  \right)+2C_{1}+2.
$$
Then
$$\left| \sum_{\lambda_{k} \leq \frac{1}{4}} h_{T}^{\epsilon}(r_{k}) - \sum_{\lambda_{k} \leq \frac{1}{4}} \frac{e^{s_{k}T}}{s_{k}} \right| = \bo_{12}(e^{3T/4}),$$
where
$$s_{1} \df \left\{ \begin{array}{cc}
\frac{1}{2} + \sqrt[+]{\frac{1}{4} - \lambda_{1}} &  \lambda_{1} \leq \frac{1}{4},   \\
\infty & \lambda_{1} > \frac{1}{4}.  \end{array}\right. 
$$
\el
\pf
First we deal with the case  $\lambda_{0}=0.$ We have
$$h_{T}^{\epsilon}(r_{0})=h_{T}^{\epsilon}(\frac{i}{2})=\frac{2}{i}\sin(iT) \widehat{\phi}(i \epsilon) + 2T =\widehat{\phi}(i \epsilon)(e^{T}-e^{-T})+2T =(1+\bo_{1}(\epsilon))(e^{T}-e^{-T})+2T.    $$
Hence, $$|h_{T}^{\epsilon}(r_{0})-e^{T}| \leq C_{1}\epsilon e^{T} + e^{-T}|\bo_{1}(\epsilon)| + 2T \leq C_{1}\epsilon e^{T} + 1+C_{1}\epsilon + 2T \leq (2C_{1}+2)e^{3T/4}. $$ The last inequality follows from the trivial bound $1+2T \leq 2e^{3T/4},$ and by noting that $\epsilon = e^{-T/4}.$

Next we handle the (possible) case of nonzero small eigenvalues. Let $0 < k \leq \sml,$
$$\lambda = \lambda_{k}, \quad r = r_{k}, \quad s=s_{k}.$$
It follows that
$$h_{T}^{\epsilon}(r) = \frac{\widehat{\phi}(i\epsilon s)}{s} \left(e^{sT}-e^{-sT} \right) + \frac{\widehat{\phi}(i\epsilon (s-1))}{s-1} \left(e^{(s-1)T}-e^{-(s-1)T} \right). $$
Here, $ \frac{1}{2} \leq s \leq s_{1} < 1.$ By Lemma~\ref{lemBump},~Part~\ref{itPhiC},
$$\widehat{\phi}(i\epsilon s) = 1+\bo_{1}(\epsilon), ~\text{and}~ \widehat{\phi}(i\epsilon (s-1)) = 1+\bo_{1}(\epsilon). $$
Therefore,
\begin{multline*}
\left| h_{T}^{\epsilon}(r) - \frac{e^{sT}}{s} \right| \leq  \frac{e^{sT}}{s} + (1+|\bo_{1}(\epsilon)|)e^{-sT} + \frac{1}{|s-1|}(1+e^{T/2})(1+|\bo_{1}(\epsilon)|) \leq  \\ \frac{1}{\frac{1}{2}} e^{T}C_{1}\epsilon + (1+C_{1}\epsilon)+\frac{1}{1-s_{1}}(1+e^{T/2})(1+C_{1}\epsilon).
\end{multline*}
Next, set $\epsilon = e^{-T/4}.$ We obtain
$$\left| h_{T}^{\epsilon}(r) - \frac{e^{sT}}{s} \right| \leq \left(1+3C_{1} + \frac{2}{1-s_{1}}(1+C_{1})  \right)e^{3T/4}. $$  Then, finally, we have
$$\left| \sum_{\lambda_{k} \leq \frac{1}{4}} h_{T}^{\epsilon}(r_{k}) - \sum_{\lambda_{k} \leq \frac{1}{4}} \frac{e^{s_{k}T}}{s_{k}} \right| \leq \left[ (\sml-1)\left(1+3C_{1} + \frac{2}{1-s_{1}}(1+C_{1})  \right)+2C_{1}+2 \right]e^{3T/4}.$$
Of course, if there are no nonzero small eigenvalues, $\sml=1.$
\epf

\subsection{Application of the trace formula}
Let $\Gamma$ be a cofinite Fuchsian group, and let $M = \Gamma \setminus \HH $ be the corresponding hyperbolic-orbifold. Let $\cgs(M)$ denote the set of closed geodesics of $M,$ and let $\pgs(M)$ denote the set of  prime (or primitive) closed geodesics (see \cite[page 245]{Buser}). For each $\gamma \in \cgs(M)$ there exists a unique prime geodesic $\gamma_{0}$ and a unique exponent $m \geq 1$ so that $\gamma = \gamma_{0}^{m}.$ Let $l(\gamma)$ denote the \emph{length} of $\gamma.$ Associated to $\gamma$ is a unique hyperbolic conjugacy class $\{ P_{\gamma} \}_{\Gamma} $ with norm
$$N_{\gamma} \df N(P_{\gamma}) = e^{l(\gamma)}.  $$

For $\gamma \in \cgs(M)$ define
$$\Lambda(\gamma) \df \log{N_{\gamma_{0}}} = l(\gamma_{0}).  $$

Let
$$H_{\epsilon}(T)=\sum_{\gamma \in \cgs(M)} \frac{\Lambda(\gamma)}{N_{\gamma}^{1/2}-N_{\gamma}^{-1/2}}g_{T}^{\epsilon}(x),
$$
$$H(T)=\sum_{l(\gamma) \leq T} \Lambda(\gamma) \frac{1+N_{\gamma}^{-1}}{1-N_{\gamma}^{-1}}. $$

Note that $H_{\epsilon}(T)$ is an approximation of $H(T),$ and that for any $\epsilon > 0,$ (\cite[page 298]{Randol2})
\beq \label{eqHapprox}
H_{\epsilon}(T-\epsilon) \leq H(T) \leq H_{\epsilon}(T+\epsilon).
\eeq

An application of the Selberg trace formula yields
\begin{multline} \label{stfHub}
\sum_{\lambda_{n}\leq \frac{1}{4}}h_{T}^{\epsilon}(r_{n}) + \sum_{\lambda_{n} > \frac{1}{4}}h_{T}^{\epsilon}(r_{n})
+ \frac{1}{4\pi} \int_{\RR} h_{T}^{\epsilon}(r) \frac{-\phi'}{\phi}(\frac{1}{2}+ir)~dr = \frac{|\F|}{4\pi} \int_{\RR} rh_{T}^{\epsilon}(r)\tanh(\pi r)~dr  \\ + H_{\epsilon}(T) + \sum_{\{R\}} \frac{1}{2m_{R} \sin \theta_{R}}\int_{\RR} \frac{e^{-2r \theta_{R} }}{1+e^{-2\pi r}} h_{T}^{\epsilon}(r)~dr +
 \frac{h_{T}^{\epsilon}(0)}{4} \tr(I-\mathfrak{S}(\frac{1}{2})) \\ -\tau g_{T}^{\epsilon}(0) \log 2 - \frac{\tau}{2\pi } \int_{\RR} h_{T}^{\epsilon}(r) \psi(1+ir)~dr.
\end{multline}

Our next goal is to let $\epsilon = e^{-T},$ and estimate each term, in terms of the variable $T.$

Recall from Theorem~\ref{thmSpecBound} that
$$\mu(r) \leq C \left(r^{2}+\frac{1}{4} \right) \quad (r \geq 0).$$
%:randolSpec
\bl \label{lemHighSpec}
Let $\epsilon=e^{-T/4},$ $$ C_{13}=\frac{41}{6} C C_{10}. $$ Then
$$\left|\sum_{\lambda_{n} > \frac{1}{4}}h_{T}^{\epsilon}(r_{n})
+ \frac{1}{4\pi} \int_{\RR} h_{T}^{\epsilon}(r) \frac{-\phi'}{\phi}(\frac{1}{2}+ir)~dr\right| = \bo_{13}(e^{3T/4}).  $$
\el
\pf
\begin{multline*}
\left|\sum_{\lambda_{n} > \frac{1}{4}}h_{T}^{\epsilon}(r_{n})
+ \frac{1}{4\pi} \int_{\RR} h_{T}^{\epsilon}(r) \frac{-\phi'}{\phi}(\frac{1}{2}+ir)~dr\right| \leq \\ \sum_{\lambda_{n} > \frac{1}{4}}|h_{T}^{\epsilon}(r_{n})|
+ \frac{1}{4\pi} \int_{\RR} \left|h_{T}^{\epsilon}(r)\right| \left|\frac{\phi'}{\phi}(\frac{1}{2}+ir)\right|~dr = \int_{0}^{\infty}\left|h_{T}^{\epsilon}(r)\right|~d\mu(r).
\end{multline*}
By Theorem~\ref{thmSpecBound}, $\mu(r) \leq C(r^{2}+\frac{1}{4}).$ The result now follows from Lemma~\ref{lemParts}.
\epf

%:randolParabolic
\bl \label{lemHubPar}
Let $\epsilon=e^{-T/4},$ $$C_{14}= C_{10}\frac{296\tau}{3\pi}+C_{10} \frac{\tau }{2}+2\tau \log{2},$$ then
$$\frac{\tau}{2\pi} \int_{\RR} |h_{T}^{\epsilon}(r) \psi(1+ir)|~dr+\left|\frac{h_{T}^{\epsilon}(0)}{4} \tr(I-\mathfrak{S}(\frac{1}{2}))\right| +| \tau g_{T}^{\epsilon}(0) \log 2| =\bo_{14}(e^{3T/4}). $$
\el
\pf
We previously saw (equation (\ref{eqDiGammaBound})) that for $r \in \RR,$
$$|\psi(1+ir)| \leq 4|r|^{1/4} + 4.$$
Hence $|\psi(1+ir)| \leq 8(1+|r|),$ and
$$
\frac{\tau}{2\pi} \int_{-\infty}^{\infty} |h_{T}^{\epsilon}(r) \psi(1+ir)|~dr \leq \frac{\tau}{2\pi} \int_{-\infty}^{\infty} |h_{T}^{\epsilon}(r)|8(1+|r|)~dr =\frac{\tau}{2\pi} \int_{0}^{\infty} |h_{T}^{\epsilon}(r)|~du(r),
$$
where
$$u(r)=\int_{-r}^{r}8(1+|x|)~dx=16(\frac{r^{2}}{2}+r) \leq 16(\frac{3}{2}r^{2}+1) \quad (r\in\RR).
$$
An application of Lemma~\ref{lemParts} yields
$$\frac{\tau}{2\pi} \int_{0}^{\infty} |h_{T}^{\epsilon}(r)|~du(r) \leq C_{10}\frac{296\tau}{3\pi}\frac{e^{T/2}}{\epsilon}.$$

Next, an application of Lemma~\ref{lemHest}, the inequalities $|\tr(I-\mathfrak{S}(\frac{1}{2}))|\leq 2\tau$ and $|g_{T}^{\epsilon}(0)| \leq 2,$ imply
$$\left|\frac{h_{T}^{\epsilon}(0)}{4} \tr(I-\mathfrak{S}(\frac{1}{2}))\right| +| \tau g_{T}^{\epsilon}(0) \log 2| \leq \frac{\tau C_{10}}{2}e^{T/2}+2\tau \log{2} \leq \left( \frac{\tau C_{10}}{2}+2\tau \log{2} \right) e^{3T/4}. $$
\epf

%:randolElliptic
\bl \label{lemHubEli}
Let $\epsilon=e^{-T/4},$
$$C_{15} =  \frac{56C_{10}}{3}\left|\sum_{\{R\}} \frac{1}{2m_{R} \sin \theta_{R}}\right|,$$
then
$$\left|\sum_{\{R\}} \frac{1}{2m_{R} \sin \theta_{R}}\int_{\RR} \frac{e^{-2r \theta_{R} }}{1+e^{-2\pi r}} h_{T}^{\epsilon}(r)~dr\right| = \bo_{15}(e^{3T/4}).  $$
\el
\pf
Since $$\left|\frac{e^{-2r \theta_{R} }}{1+e^{-2\pi r}}\right| \leq 1, $$
$$\left|\sum_{\{R\}} \frac{1}{2m_{R} \sin \theta_{R}}\int_{\RR} \frac{e^{-2r \theta_{R} }}{1+e^{-2\pi r}} h_{T}^{\epsilon}(r)~dr\right| \leq \left|\sum_{\{R\}} \frac{1}{2m_{R} \sin \theta_{R}}\right| \int_{0}^{\infty}|h_{T}^{\epsilon}(r)|~du(r), $$
where $u(r)=2r.$
Hence $|u(r)|\leq 2r^{2}+2,$ and applying Lemma~\ref{lemParts} yields
$$\left|\sum_{\{R\}} \frac{1}{2m_{R} \sin \theta_{R}}\int_{\RR} \frac{e^{-2r \theta_{R} }}{1+e^{-2\pi r}} h_{T}^{\epsilon}(r)~dr\right| \leq \left( \frac{56C_{10}}{3}\left|\sum_{\{R\}} \frac{1}{2m_{R} \sin \theta_{R}}\right|\right)e^{3T/4}. $$
\epf

%:randolIdentity
\bl \label{lemHubId}
$$\frac{|\F|}{4\pi}\int_{-\infty}^{\infty}|h_{T}^{\epsilon}(r)\,r\tanh(\pi r)|~dr \leq \frac{3}{2\pi} |\F|C_{10} \frac{e^{T/2}}{\epsilon}.$$
\el
\pf
Since $\,r\tanh(\pi r) \,$ is even, it suffices to estimate
$$
\frac{|\F|}{2\pi}\int_{0}^{\infty}|h_{T}^{\epsilon}(r)\,r\tanh(\pi r)|~dr \leq \frac{|\F|}{2\pi}\int_{0}^{\infty}|h_{T}^{\epsilon}(r)|\,r~dr = \frac{|\F|}{2\pi}\int_{0}^{\infty}|h_{T}^{\epsilon}(r)|~du(r),
$$
where $u(r)=\int_{0}^{r}x~dx = \frac{r^{2}}{2}.$ (Note that we used the trivial estimate $|\tanh(\pi r)| \leq 1.$) The result now follows from Lemma~\ref{lemParts}.
\epf

%:randolHEpsilon
Our main result for this section is
\bl \label{lemFinishProof}
Let $\epsilon = e^{-T/4},$ and let
$$C_{16}= C_{12}+C_{13}+C_{14}+ C_{15} + \frac{3}{2\pi} |\F|C_{10}. $$
Then $$H_{\epsilon}(T)=\sum_{\lambda_{k} \leq \frac{1}{4} }\frac{e^{T s_{k}}}{s_{k}}+\bo_{16}(e^{3T/4}).$$
\el
\pf
The proof follows immediately from equation (\ref{stfHub}), and Lemmas \ref{lemSmallEigBounds}, \ref{lemHighSpec}, \ref{lemHubPar}, \ref{lemHubEli}, \ref{lemHubId}.
\epf

\subsection{Upper bound for the Huber constant}
In this section we prove the main result of this paper.

\subsection{Asymptotics for $H(T)$}
\bl \label{lemHofT}
Let $\epsilon=e^{-T/4},$
$$C_{17} = 4\sml + 4C_{16}, $$
$$C_{18} =  4\sml + 5C_{16}. $$
Then
\begin{enumerate}
\item $H(T) - H_{\epsilon}(T) = \bo_{17}(e^{3T/4}).  $
\item $$ H(T) = \sum_{\lambda_{k} \leq \frac{1}{4} }\frac{e^{T s_{k}}}{s_{k}} + \bo_{18}(e^{3T/4}).$$

\end{enumerate}
\el
\pf
We first prove (1). Recall from equation (\ref{eqHapprox}) that
\beq \label{eqTmp9qa}
H_{\epsilon}(T-\epsilon) \leq H(T) \leq H_{\epsilon}(T+\epsilon).
\eeq
In addition, it follows from the definition of $H_{\epsilon}(T)$  that
\beq \label{eqTmp7hc}
H_{\epsilon}(T-\epsilon) \leq H_{\epsilon}(T) \leq H_{\epsilon}(T+\epsilon).
\eeq
By Lemma~\ref{lemFinishProof} we have
$$H_{\epsilon}(T \pm\epsilon)=\sum_{\lambda_{k} \leq \frac{1}{4} }\frac{e^{(T\pm\epsilon) s_{k}}}{s_{k}}+\bo_{16}(e^{3(T\pm\epsilon)/4}).$$
Now, setting $\epsilon=e^{-T/4},$ it follows from elementary estimates that
$$|H_{\epsilon}(T+\epsilon)-H_{\epsilon}(T-\epsilon)| \leq 4\sml e^{3T/4} + 4C_{16}e^{3T/4}. $$ Part (1) now follows from \eqref{eqTmp9qa} and \eqref{eqTmp7hc}. Part (2) follows immediately from Lemma~\ref{lemFinishProof}.
\epf

\subsection{The Chebyshev function} \label{secCheb}
Define $$\Psi(T) \df \sum_{l(\gamma)\leq T} \Lambda(\gamma), $$ where $\Lambda(\gamma) = l(\gamma_{0}).$ Since $\Gamma$ is discrete, there exists a hyperbolic conjugacy class $\{ P_{s} \}$   with minimal norm $N(P_{s}) > 1.$ Choose $c > 0$ so that $$1 < c < N(P_{s}). $$

\bl
Let
$$C_{19}=C_{18} +\frac{4}{1-c^{-1}}(2\sml+C_{18}).  $$ Then
$$\Psi(T) = \sum_{\lambda_{k} \leq \frac{1}{4}}\frac{e^{Ts_{k}}}{s_{k}} + \bo_{19}(e^{3T/4}).$$
\el
\pf
By Lemma~\ref{lemHofT},
$$H(T) = \sum_{\lambda_{k} \leq \frac{1}{4}}\frac{e^{T s_{k}}}{s_{k}} +\bo_{18}(e^{3T/4}). $$
Since $$H(T)=\sum_{l(\gamma) \leq T} \Lambda(\gamma) \frac{1+N_{\gamma}^{-1}}{1-N_{\gamma}^{-1}}, $$
we can rewrite $H(T)$ as
$$H(T) = \Psi(T) + \sum_{l(\gamma)\leq T} \Lambda(\gamma) \frac{2N_{\gamma}^{-1}}{1-N_{\gamma}^{-1}}.  $$
Both of the above terms are positive for all $T.$ Thus, from Lemma~\ref{lemHofT} and a crude estimate (for each small eigenvalue, $\frac{1}{2} \leq s_{k} \leq 1,$ hence $e^{s_{k}t} \leq e^{t}.$), it follows that
\beq
\Psi(T) \leq H(T) \leq (2\sml+C_{18})e^{T} \label{eqTmpjh29}.
\eeq
But
\begin{multline*}
\sum_{l(\gamma)\leq T} \Lambda(\gamma) \frac{2N_{\gamma}^{-1}}{1-N_{\gamma}^{-1}} \leq \frac{2}{1-c^{-1}}\sum_{l(\gamma)\leq T} \Lambda(\gamma) N_{\gamma}^{-1}=  \frac{2}{1-c^{-1}}\int_{0}^{T} e^{-x}~d\Psi(x) \leq \\\frac{2}{1-c^{-1}}(2\sml+C_{18})(1+T).
\end{multline*}
The last inequality is obtained by integrating by parts and using \eqref{eqTmpjh29}.  Next, using the crude bound
$(1+T)\leq 2e^{3T/4}, $ we obtain
$$\sum_{l(\gamma)\leq T} \Lambda(\gamma) \frac{2N_{\gamma}^{-1}}{1-N_{\gamma}^{-1}} \leq \frac{4}{1-c^{-1}}(2\sml+C_{18})e^{3T/4}. $$ Finally
$$\left|\Psi(T)-\sum_{s_{k}}\frac{e^{T s_{k}}}{s_{k}}\right| \leq (C_{18}+\frac{4}{1-c^{-1}}(2\sml+C_{18}))e^{3T/4}. $$
\epf

\subsection{The Huber constant} \label{secMainResult}
Let
$$\Theta(T)=\sum_{ \substack{l(\gamma)\leq T \\ \gamma ~ \text{primitive}}}\Lambda(\gamma). $$
Choose $\mu > 0$ so that $\mu < l(\gamma)$ for all $\gamma \in \cgs(M).$
\bl \label{lemThetaGrowth}
$$\Theta(T) = \sum_{\lambda_{k} \leq \frac{1}{4}}\frac{e^{T s_{k}}}{s_{k}} + \bo_{20}(e^{3T/4}),$$ where
$$C_{20}=C_{19} + \frac{4}{\mu}(2\sml+C_{18}). $$
\el
\pf
Let $m(T)=[T/\mu],$ then (\cite[page 251]{Buser})
$$
\Psi(T) - \Theta(T) = \sum_{m=2}^{m(T)}\Theta(T/m) \leq m(T)\Theta(T/2) \leq \frac{T}{\mu}\Psi(T/2).
$$
By \eqref{eqTmpjh29},

$$\frac{T}{\mu}\Psi(T/2) \leq  \frac{1}{\mu} Te^{T/2}(2\sml+C_{18}) \leq \frac{4}{\mu} e^{3T/4}(2\sml+C_{18}). $$
Hence $$ \left|\Theta(T) - \Psi(T) \right| \leq \frac{4}{\mu} e^{3T/4}(2\sml+C_{18}). $$ The lemma now follows.
\epf
%:small theta
Let $x = e^{T},$ and let
$$
\theta(x) \df \sum_{ \substack{N_{\gamma} \leq x \\ \gamma ~ \text{primitive}}} \Lambda(\gamma).
$$ With this change of variable, Lemma~\ref{lemThetaGrowth} becomes
$$\theta(x) = \sum_{\lambda_{k} \leq \frac{1}{4} }\frac{x^{s_{k}}}{s_{k}} + \bo_{20}(x^{3/4}).
$$
Let
$$\li(x) = \int_{2}^{x}\frac{dy}{\log y},$$ and
recall the definition of the constant $c$ (\S\ref{secCheb}). We have
\bl \label{lemIntLog}
Let
$$ C_{21} = |c-2| \frac{1}{\log{2}} + |2-\sqrt{c}| \frac{2}{\log{c}}. $$
Then for $s \in [\frac{1}{2},1],$
$$\left|\int_{c}^{x}\frac{y^{s-1}}{\log y}~dy - \li(x^{s}) \right| \leq C_{21}.$$
\el
\pf
An elementary argument shows that
$$
\left|\int_{c}^{x}\frac{y^{s-1}}{\log y}~dy - \li(x^{s}) \right| = \left|\int_{c^{s}}^{2}\frac{dy}{\log y} \right|.
$$
Now, if $c^{s} \geq 2,$ then
$$\left|\int_{c^{s}}^{2}\frac{dy}{\log y} \right| \leq (c^{s}-2) \frac{1}{\log{2}} \leq  (c-2) \frac{1}{\log{2}}.
$$
If $c^{s} < 2,$ then
$$\left|\int_{c^{s}}^{2}\frac{dy}{\log y} \right| \leq (2 - c^{s}) \frac{1}{\log{c^{s}}} \leq  (2-c^{1/2}) \frac{2}{\log{c}}. $$ Hence, in either case,
$$\left|\int_{c^{s}}^{2}\frac{dy}{\log y} \right| \leq   |c-2| \frac{1}{\log{2}} + |2-c^{1/2}| \frac{2}{\log{c}}. $$
\epf

\bl \label{lemLiBound}
Let $$C_{22} = \left(1-\frac{1}{\log{2}} \right)^{-1}. $$
Then for $x>1,$
$$\li(x) \leq C_{22} \frac{x}{\log{x}}. $$
\el
\pf
Integrating by parts give us
$$\li(x) = \int_{2}^{x} \frac{dy}{\log{y}} \leq \frac{x}{\log{x}} + \int_{2}^{x} \frac{dy}{(\log{y})^{2}} \leq \frac{x}{\log{x}} + \frac{1}{\log{2}}\int_{2}^{x} \frac{dy}{\log{y}}. $$
The lemma now follows.
\epf

\begin{thm} \label{thmHuberCons} Let
$$C_{u} = C_{21}\sml + C_{20}\frac{c^{3/4}}{\log{c}} + C_{20} + C_{20}C_{22} + \frac{3}{4}C_{20}C_{21}. $$
Then for all $x>1,$
$$\left| \pi(x) - \sum_{\lambda_{k} \leq \frac{1}{4} }\li(x^{s_{k}}) \right|  \leq C_{u}\frac{x^{3/4}}{\log{x}}. $$
In particular, $C_{M} \leq C_{u}.$
\end{thm}
\pf
First note that
$$
\pi(x) = \int_{c}^{x} \frac{d\theta(y)}{\log y}.
$$
First, for notational convenience, set
$$f(y)  \df \theta(y) - \sum_{\lambda_{k} \leq \frac{1}{4} }\frac{y^{s_{k}}}{s_{k}} = \bo_{20}(y^{3/4}). $$
Next,
$$ \int_{c}^{x} \frac{d\theta(y)}{\log{y}} =  \sum_{\lambda_{k} \leq \frac{1}{4} } \int_{c}^{x} \frac{y^{s_{k}-1}}{\log y}~dy + \int_{c}^{x} \frac{df(y)}{\log{y}}.  $$
By Lemma~\ref{lemIntLog},
$$\sum_{\lambda_{k} \leq \frac{1}{4} } \int_{c}^{x} \frac{y^{s_{k}-1}}{\log y}~dy  =  \sum_{\lambda_{k} \leq \frac{1}{4} }\li(x^{s_{k}}) +  \sml \bo_{21}(1). $$
Now
$$\int_{c}^{x} \frac{df(y)}{\log{y}} = \frac{f(x)}{\log{x}} - \frac{f(c)}{\log{c}} + \int_{c}^{x} f(y) \frac{dy}{y(\log{y})^{2}}, $$
and (using Lemma~\ref{lemIntLog}),
\begin{multline*}
\int_{c}^{x} |f(y)| \frac{dy}{y(\log{y})^{2}} \leq C_{20}\int_{c}^{x} y^{3/4} \frac{dy}{y(\log{y})^{2}} = C_{20} \left( \frac{c^{3/4}}{\log{c}} + \frac{x^{3/4}}{\log{x}} + \frac{3}{4}\int_{c}^{x} y^{(\frac{3}{4} - 1)} \frac{dy}{\log{y}}  \right)  \\ \leq   C_{20} \left( \frac{c^{3/4}}{\log{c}} + \frac{x^{3/4}}{\log{x}} + \frac{3}{4}( \li(x^{3/4}) + C_{21} ) \right).
\end{multline*}
But by Lemma~\ref{lemLiBound},
$$\frac{3}{4} \li(x^{3/4}) \leq C_{22} \frac{x^{3/4}}{\log{x}}. $$
Hence, putting everything together, for $x > 1,$
\begin{multline*}
 \left| \pi(x) - \sum_{\lambda_{k} \leq \frac{1}{4} }\li(x^{s_{k}}) \right|  \leq C_{21}\sml + C_{20}\frac{c^{3/4}}{\log{c}} + C_{20}\frac{x^{3/4}}{\log{x}} + C_{20}C_{22}\frac{x^{3/4}}{\log{x}} + \frac{3}{4}C_{20}C_{21}.
\end{multline*}
And using the trivial inequality $1 < \frac{x^{3/4}}{\log{x}}, $ we obtain
$$\left| \pi(x) - \sum_{\lambda_{k} \leq \frac{1}{4} }\li(x^{s_{k}}) \right|  \leq C_{u}\frac{x^{3/4}}{\log{x}}, $$
where $$C_{u} = C_{21}\sml + C_{20}\frac{c^{3/4}}{\log{c}} + C_{20} + C_{20}C_{22} + \frac{3}{4}C_{20}C_{21}. $$
\epf

\subsection{The Huber constant for $\PSL(2,\ZZ)$} \label{secHubPSL}
For $\Gamma = \PSL(2,\ZZ),$ we can take $c = 6.85.$ Also, $\Gamma$ has no small eigenvalues (except for the trivial one), so $\sml = 1.$ Using the computations in \S\ref{secMod},  an explicit value for $C_{u}$ can be calculated:
\begin{thm} \label{thmPSLHub}
Let $\Gamma = \PSL(2,\ZZ).$ Then $C_{u} = 16,\!607,\!349,\!020,\!658$ is an upper bound for the Huber constant.
\end{thm}

Using the theory laid out in \cite{Sarnak2}, and the excellent open-source computational program \textsc{pari/gp} \cite{PARI2}, we can explicitly calculate (and list out) the length spectrum for $\PSL(2,\ZZ).$ Our computations suggest that $$C_{M} < 2. $$ So it seems that our result is probably not sharp!

\section*{Computer experiments for $\PSL(2,\ZZ)$}

Let $\Gamma_0 = \PSL(2,\ZZ),$ and  $\Gamma$ be a finite index subgroup of  $\Gamma_0.$ Let $\{ P_0 \} $ denote the set of primitive hyperbolic conjugacy classes of $\Gamma.$ We compute length spectrum (norms of primitive hyperbolic conjugacy classes) using the ideas of \cite{Sarnak2}.

Let $Q(x,y)=ax^2+bxy+cy^2$ be a primitive, indefinite binary quadratic form of discriminant $d=b^2-4ac > 0.$ Two forms $[a,b,c]$ and $[a',b',c']$ are called equivalent (in the narrow sense) if they are related by a unimodular transformation.

Consider the map $\phi$ which sends the form $[a,b,c]$ to the matrix
$$
\left(\begin{array}{cc}
\frac{t_d-b u_d}{2} & -cu_d\\
a u_d & \frac{t_d+bu_d}{2} \end{array}\right)
$$
where $t_d,u_d > 0$ is the fundamental solutions of Pell's equation $t^2-du^2=4.$ Set $$\epsilon_d = \frac{t_d+\sqrt{d}~u_d}{2}. $$

Let $\mathcal{D}$ be the set of fundamental discriminants $ \{ d >0, d \equiv 0,1 \mod 4 \} $ and let $h(d)$ denote the number of inequivalent primitive forms of discriminant $d.$ For each $d \in \mathcal{D}$ let $Q_1(d), \dots Q_{h(d)}(d) $ be a complete set of inequivalent forms of discriminate $d.$ Finally, set
$$QF=\bigcup_{d\in \mathcal{D}} \{ Q_1(d), \dots Q_{h(d)}(d) \} $$

\bl(\text{\rm \cite{Sarnak2}})
The set $ \phi(QF) $ is the set of conjugacy classes of primitive hyperbolic elements of $\PSL(2,\ZZ).$ The norm of each hyperbolic element $\phi(Q_i(d))$ is $\epsilon_d^2.$
\el

Using the above  ideas, one can compute the primitive conjugacy classes for $\Gamma_0 = \PSL(2,\ZZ).$ We next explain how to compute the length spectrum for $\Gamma,$ a finite-index normal subgroup of $\Gamma_0.$

Let $\Gamma \lhd \Gamma_0 $ be a finite-index normal subgroup of $\Gamma_0 = \PSL(2,\ZZ),$  $G=\Gamma_0 / \Gamma, $ and $n=|G|.$ For each primitive conjugacy class $\{P_\gamma\}$ of $\Gamma_0$ let $m$ represent the order of $P_\gamma \in G.$ It follows that $\{P_\gamma^m\}$ is a subset of $\Gamma$ that splits into conjugacy classes with respect to $\Gamma.$

\bl (\text{\rm \cite{Sarnak2}})
The number of conjugacy classes into which $ \{ P_\gamma^m  \}$ splits in $\Gamma$ is $k=n/m=|G|/m.$
\el

\section{Algorithm for computing primitive conjugacy classes}
Let $\Gamma \lhd \Gamma_0, $ $x>0.$ We wish to compute the norms and multiplicities of all primitive hyperbolic conjugacy classes  $\{P_\gamma \}$ with norm $N( \{P_\gamma \} ) \leq x.$

Step 1. Compute the primitive conjugacy classes for the modular group $\Gamma_0$ of norm $x$ or less. Since $\epsilon_d^2$ is the norm of a prime element that is mapped by the function $\phi$ it follows that $\sqrt{x} \geq \epsilon_d > \sqrt{d}/2. $ Thus it suffices to consider discriminants $d < 4x$ in order to \textbf{guarantee} that all norms less than $x$ are achieved. In other words, if $d>4x$ then $\epsilon_d^2 > x.$

For each $d \in \mathcal{D}_x = \{ d\in \mathcal{D}~|~\epsilon_d^2 \leq x \},$
compute $h(d)$ and representatives of inequivalent quadratic forms of discriminant $d.$ For each representative $Q_i(d)$ compute $\phi(Q_i(d))$ in order to obtain representatives for $\{P_\gamma \}$ in $\Gamma_0.$

Step 2. From step 1, let $z = P_\gamma$ be a representative of a class in $\Gamma_0$ of norm $\epsilon_d^2.$ Compute the order $m$ of $z$ in $G=\Gamma_0 / \Gamma.$ Then $ \{ P_\gamma^m \}$ splits into $k=|G|/m$ primitive conjugacy classes each with norm $\epsilon_d^{2m}.$

Step 3. From step 2, we count the primitive conjugacy classes with norm  $\epsilon_d^{2m} \leq x.$

One thing we can do is verify the prime geodesic theorem for the group $\Gamma.$ We can implement these algorithms using the computer package PARI/GP.

\subsection*{Pari/GP Algorithm}

\begin{description}
\item $IsProp(D)=if ( (D \% 4 == 1 || (D) \% 4==0) \quad \&\& \, \quad issquare(D), 0, 1);$
\item $QFBclassno(D) = qfbclassno(D) * if (D < 0 \quad || \quad  norm(quadunit(D)) < 0, 1, 2);$
\item $QUADunit(D)=quadunit(D)*if(D<0 \quad || \quad norm(quadunit(D))<0,quadunit(D),1); $
\item $Li(x) = -eint1(log(1/x));$
\item $MaxD=10000;$
\item \{
\item $for (D=5,MaxD, if (IsProp(D)==0,$
\item $n=QUADunit(D) \wedge 2.0;$
\item $if (n < MaxD,$
\item $m=QFBclassno(D);$
\item $q=Li(n);$
\item $write("FILE.csv", n "," m "," D "," q ) ) ) );$
\item \}
\end{description}

\newpage

\begin{table}
\caption{Data for the modular group}
\begin{tabular}{|l|l|l|l|l|l|l|l|l|l|l|}
\hline
1&2&3&4&5&6&7&8&9&10&11\\
\hline
6.85&1&5&4.68&1&4.68&3.68&3.05&1.21&2.62&1.41\\
\hline
13.93&2&12&7.75&3&2.58&4.75&4.44&1.07&3.73&1.27\\
\hline
22.96&2&21&10.87&5&2.17&5.87&5.92&.99&4.79&1.23\\
\hline
33.97&2&32&14.17&7&2.02&7.17&7.49&.96&5.83&1.23\\
\hline
33.97&1&8&14.17&8&1.77&6.17&7.49&.82&5.83&1.06\\
\hline
46.98&2&45&17.69&10&1.77&7.69&9.15&.84&6.85&1.12\\
\hline
61.98&4&60&21.45&14&1.53&7.45&10.87&.68&7.87&0.95\\
\hline
78.99&2&77&25.45&16&1.59&9.45&12.68&.75&8.89&1.06\\
\hline
97.99&2&24&29.69&18&1.65&11.69&14.55&.80&9.9&1.18\\
\hline
97.99&4&96&29.69&22&1.35&7.69&14.55&.53&9.9&0.78\\
\hline
118.99&1&13&34.17&23&1.49&11.17&16.48&.68&10.91&1.02\\
\hline
118.99&2&117&34.17&25&1.37&9.17&16.48&.56&10.91&0.84\\
\hline
141.99&4&140&38.9&29&1.34&9.9&18.48&.54&11.92&0.83\\
\hline
166.99&4&165&43.86&33&1.33&10.86&20.53&.53&12.92&0.84\\
\hline
193.99&2&48&49.06&35&1.40&14.06&22.65&.62&13.93&1.01\\
\hline
193.99&4&192&49.06&39&1.26&10.06&22.65&.44&13.93&0.72\\
\hline
223&4&221&54.49&43&1.27&11.49&24.82&.46&14.93&0.77\\
\hline
254&2&28&60.15&45&1.34&15.15&27.04&.56&15.94&0.95\\
\hline
254&4&252&60.15&49&1.23&11.15&27.04&.41&15.94&0.7\\
\hline
287&4&285&66.04&53&1.25&13.04&29.31&.45&16.94&0.77\\
\hline
322&2&80&72.17&55&1.31&17.17&31.63&.54&17.94&0.96\\
\hline
322&1&20&72.17&56&1.29&16.17&31.63&.51&17.94&0.9\\
\hline
322&4&320&72.17&60&1.20&12.17&31.63&.38&17.94&0.68\\
\hline
359&4&357&78.51&64&1.23&14.51&34&.43&18.95&0.77\\
\hline
398&2&44&85.08&66&1.29&19.08&36.42&.52&19.95&0.96\\
\hline
398&8&396&85.08&74&1.15&11.08&36.42&.30&19.95&0.56\\
\hline
439&2&437&91.88&76&1.21&15.88&38.88&.41&20.95&0.76\\
\hline
482&4&120&98.89&80&1.24&18.89&41.39&.46&21.95&0.86\\
\hline
482&8&480&98.89&88&1.12&10.89&41.39&.26&21.95&0.5\\
\hline
527&4&525&106.12&92&1.15&14.12&43.94&.32&22.96&0.62\\
\hline
574&4&572&113.57&96&1.18&17.57&46.53&.38&23.96&0.73\\
\hline
623&2&69&121.23&98&1.24&23.23&49.16&.47&24.96&0.93\\
\hline
\end{tabular}
\end{table}
\begin{small}
Key: Length spectrum calculations for the modular group.
\begin{enumerate}
\item $x,$ the actual primitive length spectrum (norms); when norms repeat, they come from different quadratic forms
\item multiplicity of given norm
\item discriminant of corresponding quadratic form
\item $\text{li}(x)$, the integral logarithm of $x$ where $x$ is equal to the value of column 1
\item $\pi(x),$ a running total of the number of norms less than $x$ where $x$ is equal to the value of column 1
\item $\frac{\pi(x)}{\text{li}(x)};$ this column verifies the prime geodesic theorem
\item $|\text{li}(x) - \pi(x)|$
\item  ${x^{3/4}}/{ (\log(x))^{1/2}}$
\item the ratio of (7)/(8)
\item $x^{1/2},$ this is the conjectured true error term
\item the ratio of (7)/(10)
\end{enumerate}
\end{small}

\bibliography{hubConst}
\bibliographystyle{amsalpha}

\vspace{5mm}

\noindent
Joshua S. Friedman \\
Department of Mathematics and Sciences \\
\textsc{United States Merchant Marine Academy} \\
300 Steamboat Road \\
Kings Point, NY 11024 \\
U.S.A. \\
e-mail: FriedmanJ@usmma.edu, joshua@math.sunysb.edu, CrownEagle@gmail.com

\vspace{5mm}
\noindent
Jay Jorgenson \\
Department of Mathematics \\
The City College of New York \\
Convent Avenue at 138th Street \\
New York, NY 10031
U.S.A. \\
e-mail: jjorgenson@mindspring.com

\vspace{5mm}

\noindent
J\"urg Kramer \\
Institut f\"ur Mathematik \\
Humboldt-Universit\"at zu Berlin \\
Unter den Linden 6 \\
D-10099 Berlin \\
Germany \\
e-mail: kramer@math.hu-berlin.de

\end{document}